\documentclass[smallextended]{svjour3}

\usepackage{color}

\smartqed
\usepackage{graphicx}
\usepackage{mathptmx}      
\usepackage{listings}
\usepackage{amssymb,amsmath,  amsfonts}
\usepackage{a4wide}
\newtheorem{assumption}{Assumption}

\newcommand{\Gr}{{\rm gph\,}}

\newcommand{\xb}{\bar x}
\newcommand{\yb}{\bar y}

\newcommand{\la}{\lambda}
\newcommand{\lab}{\bar\lambda}
\newcommand{\ob}{{\bar\omega}}

\newcommand{\R}{\mathbb{R}}
\newcommand{\N}{\mathbb{N}}
\newcommand{\norm}[1]{\|#1\|}

\newcommand{\dist}[1]{{\rm d}(#1)}
\newcommand{\B}{{\cal B}}

\newcommand{\mv}{\,\vert\, }
\newcommand{\setto}[1]{\mathop{\to}\limits^#1}

\begin{document}

\title{Lipschitz and H\"older stability of optimization problems and generalized equations\thanks{The final publication is available at Springer via http://dx.doi.org/10.1007/s10107-015-0914-1}}
\author{Helmut Gfrerer \and
             Diethard Klatte}
\institute{Helmut Gfrerer\at
Institute of Computational Mathematics, Johannes Kepler University Linz,
              A-4040 Linz, Austria,
             \email{helmut.gfrerer@jku.at}
             \and
             Diethard Klatte\at
             Department of Business Administration, Professorship Mathematics for Economists, University of Zurich, Switzerland, \email{diethard.klatte@business.uzh.ch}}

\date{}

\maketitle
\begin{abstract} This paper studies stability aspects of solutions of parametric mathematical programs and generalized equations, respectively, with disjunctive constraints. We present sufficient conditions that, under some constraint qualifications ensuring metric subregularity of the constraint mapping, continuity results of upper Lipschitz and upper H\"older type, respectively, hold. Furthermore, we apply the above results to parametric mathematical programs with equilibrium constraints and demonstrate, how some classical results for the nonlinear programming problem can be recovered and even improved by our theory.
\end{abstract}

\section{Introduction}
Consider the optimization problem
\begin{equation}\label{EqParOptProb}
P(\omega) \qquad\min_x f(x,\omega)\quad\mbox{subject to}\quad q(x,\omega)\in P,
\end{equation}
depending on the parameter vector $\omega$ belonging to some
topological space $\Omega$.  In  \eqref{EqParOptProb},
$f:\R^n\times\Omega\to\R$ and $q:\R^n\times \Omega\to\R^m$ are
continuous mappings and $P\subset \R^m$ is the union of finitely
many convex polyhedra $P_i,$ $i=1,\ldots,p,$ having the representation
\begin{equation}\label{EqReprPolyhedra}
P_i=\{y\mv a_{ij}^Ty\leq b_{ij}, j=1,\ldots,m_i\}
\end{equation}
with $a_{ij}\in\R^m$ and $b_{ij}\in\R$.
We will study stability results of Lipschitz or H\"older type for stationary and optimal solutions of $P(\omega)$ if $\omega$ varies near some reference parameter $\bar{\omega}$.

Of course, the parameter dependent nonlinear programming problem
\[
NLP(\omega) \qquad \min f(x,\omega) \quad \mbox{subject to} \quad
g(x,\omega)\leq 0,\ h(x,\omega)=0,
\]
where $f(x,\omega):\R^n\times\R^s\to\R$,
$g:\R^n\times\R^s\to\R^{m_I}$ and $h:\R^n\times\R^s\to\R^{m_E}$,
is a special case of \eqref{EqParOptProb} with
$q(x,\omega):=(g(x,\omega),h(x,\omega))$ and
$P:=\R^{m_I}_-\times\{0\}^{m_E}$. Let us consider some more
involved examples.

\begin{example}\label{ExMPEC}
Consider the parameter dependent MPEC
\[
MPEC(\omega)\qquad \min f(x,\omega) \quad \mbox{subject to} \quad
\begin{array}[t]{rr} g(x,\omega)\leq 0,\ h(x,\omega)=0, \ \ \ {~} & \\
\left. \begin{array}{r} G_i(x,\omega)\geq 0, \ H_i(x,\omega)\geq  0  \\
       G_i(x,\omega)H_i(x,\omega)=0  \end{array} \right\} &  \! i=1,\ldots,m_C,
\end{array}
\]
where $f:\R^n\times\R^s\to\R$,
$g:\R^n\times\R^s\to\R^{m_I}$, $h:\R^n\times\R^s\to\R^{m_E}$ and $G,H:\R^n\times\R^s\to\R^{m_C}$.

The problem  $MPEC(\omega)$ fits into our setting \eqref{EqParOptProb} with
\[
q:=(g,h, -(G_1,H_1) , \ldots  - (G_{m_C},H_{m_C}) ),
\]
\[P:=\R^{m_I}_-\times\{0\}^{m_E}\times Q_{EC}^{m_C},\]
where $Q_{EC}:=\{(a,b)\in\R^2_-\mv ab=0\}$. Since $Q_{EC}$ is the union of the  convex polyhedra $\R_-\times\{0\}$ and $\{0\}\times\R_-$, $P$ is the union of $2^{m_C}$ polyhedra.
\end{example}

\begin{example}\label{ExMPVC}
Another prominent example is the mathematical program with
vanishing constraints (MPVC)
\[
MPVC(\omega)\qquad \min f(x,\omega) \quad \mbox{subject to} \quad
\begin{array}[t]{rr} g(x,\omega)\leq 0,\ h(x,\omega)=0, \ \ \ {~} & \\
\left. \begin{array}{r} G_i(x,\omega)\geq 0  \\
       G_i(x,\omega)H_i(x,\omega)\le 0  \end{array} \right\}&  \! i=1,\ldots,m_V,
\end{array}
\]
where $f:\R^n\times\R^s\to\R$,
$g:\R^n\times\R^s\to\R^{m_I}$, $h:\R^n\times\R^s\to\R^{m_E}$ and
$G,H:\R^n\times\R^s\to\R^{m_V}$. For more details on MPVCs  we
refer the reader to \cite{AchKa08,HoKa08}.

Again, the problem  $MPVC(\omega)$ can be written in the form
\eqref{EqParOptProb} with
\[
q:=(g,h, (G_1,H_1) , \ldots  (G_{m_V},H_{m_V}) ),
\]
\[P:=\R^{m_I}_-\times\{0\}^{m_E}\times Q_{VC}^{m_V},\]
where $Q_{VC}:=\{(a,b)\in\R_+\times\R\mv ab\leq 0\}$ is the union
of the two convex polyhedra $\R_+\times\R_-$ and $\{0\}\times\R_+$.
\end{example}

\if{
If $f$ is partially differentiable with respect to $x$ and
setting $F(x,\omega):=\nabla_xf(x,\omega)$, then the first order
optimality conditions at a local minimizer $x$ for $P(\omega)$ can
be written as a generalized equation
\begin{equation}\label{EqGE}
GE(\omega)\qquad0\in F(x,\omega)+\hat N(x;{\cal F}(\omega)),
\end{equation}
where  $\hat N(x;{\cal F}(\omega))$  stands for the {\em Fr\'echet
normal cone} to the set ${\cal F}(\omega)$ at $x$ and
\begin{equation}\label{EqOmega}
{\cal F}(\omega):=\{x\in\R^n\mv q(x,\omega)\in P\}
\end{equation}
denotes the feasible region of the  problem $P(\omega)$. We also
consider the generalized equation (\ref{EqGE}) for arbitrary continuous mappings
$F:\R^n\times\Omega\to\R^n$.}\fi
If $f$ is partially differentiable with respect to $x$, then the first order
optimality conditions at a local minimizer $x$ for $P(\omega)$ can
be written as a generalized equation
\[0\in \nabla_x f(x,\omega)+\hat N(x;{\cal F}(\omega)),\]
where  $\hat N(x;{\cal F}(\omega))$  stands for the {\em Fr\'echet
normal cone} to the set ${\cal F}(\omega)$ at $x$ and
\begin{equation}\label{EqOmega}
{\cal F}(\omega):=\{x\in\R^n\mv q(x,\omega)\in P\}
\end{equation}
denotes the feasible region of the  problem $P(\omega)$. We also
consider the generalized equation
\begin{equation}\label{EqGE}
GE(\omega)\qquad0\in F(x,\omega)+\hat N(x;{\cal F}(\omega)),
\end{equation}
for arbitrary continuous mappings $F:\R^n\times\Omega\to\R^n$.

Throughout this paper we will make the following assumption:
\begin{assumption}\label{AssBasic}There are neighborhoods $U$ of $\xb$ and $W$ of $\ob$ such
that $f$ and $q$ are twice partially differentiable with respect
to $x$, $F$ is partially differentiable with respect to $x$ on
$U\times W$,  $F(x,\omega)$, $q(x,\omega)$, $\nabla_x
f(x,\omega)$, $\nabla_x q(x,\omega)$, $\nabla_x F(x,\omega)$,
$\nabla_x^2 f(x,\omega)$ and $\nabla_x^2 q(x,\omega)$ are
continuous at $(\xb,\ob)$ and $\nabla_x^2 f(\cdot,\omega)$,
$\nabla_x^2q(\cdot,\omega)$ are continuous on $U$ for every
$\omega\in W$.
\end{assumption}

Given a fixed parameter  $\ob$ and  a solution $\xb$ of $P(\ob)$
respectively $GE(\ob)$, we are interested in
estimates of the distance of solutions $x$ of problem $P(\omega)$
respectively $GE(\omega)$ to $\xb$ for parameters $\omega$
belonging to some neighborhood of $\ob$.

We will present such  estimates in terms of the mappings  $e_l,\tau_l,\hat\tau_l:\Omega\to\R$, $l=1,2$, given by
\[e_l(\omega)= \norm{\nabla_x q(\xb,\omega)-\nabla_x q(\xb,\bar\omega)}+{\norm{q(\xb,\omega)-q(\xb,\bar\omega)}}^{\frac1l}\]
and
\[\tau_l(\omega):=\norm{\nabla_x f(\xb,\omega)-\nabla_x f(\xb,\ob)}+e_l(\omega),\quad \hat \tau_l(\omega):=\norm{F(\xb,\omega)- F(\xb,\ob)}+e_l(\omega).\]

The quantities $\tau_l(\omega)$ and
$\hat \tau_l(\omega)$  measure
how much the problem data at the reference point $\xb$ for the
perturbed problem $P(\omega)$ and $GE(\omega)$, respectively, differ
from that for the unperturbed problem $P(\ob)$ and $GE(\ob)$, respectively.

In case that we can bound the distance of a solution $x$ of
$P(\omega)$ ($GE(\omega)$)  to $\xb$ by the estimate
$L\tau_1(\omega)$ ($L\hat\tau_1(\omega)$), where $L$ denotes some
constant, we speak of {\em upper Lipschitz stability} of the
solutions. We speak of {\em upper H\"older stability} when a bound
of the form $\norm{x-\xb}\leq L\tau_2(\omega)$
(or $\norm{x-\xb}\leq L\hat\tau_2(\omega)$) is available. This notation
is motivated by the situation that $\Omega$ is a metric space
equipped with the metric $d$, and $q(\xb,\cdot)$, $\nabla_x
q(\xb,\cdot)$ and $\nabla_x f(\xb,\cdot)$ (or $F(\xb,\cdot)$)
are Lipschitz near $\ob$, because in this
circumstance the bounds are of the form $\norm{x-\xb}\leq
Ld(\omega,\ob)$ and $\norm{x-\xb}\leq L\sqrt{d(\omega,\ob)}$,
respectively. Note that in this case the property of upper Lipschitz stability is also called {\em isolated calmness} in the literature (see, e.g.,\cite{DoRo09}).

Many quantitative stability results are known for the parameter dependent
nonlinear programming problem $NLP(\omega)$. We refer to the
monographs \cite{BonSh00,DoRo09,KlKum02} and the references therein.
Compared with the huge amount of stability results for $NLP$, very
little research has been done with the stability of $P(\omega)$.
Most of the results are known for $MPEC(\omega)$, see e.g.
\cite{AruIz05,Iz04,JoShikStef12,SchSch00,Shik12}.
Stability of M-stationarity solutions was characterized in the recent paper \cite{CerOutPi14} for a special type of problems with complementarity constraints. Sensitivity and stability results for MPVC are given in \cite{IzSo09}. In the
recent paper \cite{GuLiYe12}, Guo, Lin and Ye  presented various
stability results for  more general problems. In particular, they
proved upper Lipschitz stability for stationary pairs consisting
of stationary solutions and associated multipliers under the
structural assumption that the graph of the limiting normal cone
mapping to the set $P$ is the union of finitely many convex
polyhedra.

In contrary to the stability results of \cite{GuLiYe12}, we focus
our interest on the stability of solutions of (\ref{EqGE})
on its own  and not of
stationary pairs. This has the advantage that our theory is also
applicable in case when  multipliers do not exist or the
multipliers do not behave continuous, cf. Examples \ref{ExDemoMPEC}
and \ref{JoShikStef}  below.

Our results are mainly based on  characterizations of metric
subregularity as introduced in \cite{Gfr11,Gfr13a,Gfr13b,Gfr14a}.
The main constraint qualifications used in this paper are that, at
the reference point $\xb$, either the first order or the second
order sufficient conditions for metric subregularity are fulfilled
for the problem $P(\ob)$. Although the property of metric
subregularity is not stable in general, we will see that the
sufficient conditions of order $l$, $l=1,2$, for metric
subregularity guarantee some stability. In particular, we will
prove that there is some constant $\gamma$ such that for all
points $x$ feasible for the problem $P(\omega)$ and satisfying
$\norm{x-\xb}>\gamma e_l(\omega)$, the constraints of $P(\omega)$
are metrically regular near $(x,0)$ with some uniform modulus.
This result allows us to divide the solution sets of $P(\omega)$, similarly for
$GE(\omega)$, into two parts: one part is contained in
a ball around $\xb$ with radius $\gamma e_l(\omega)$ and behaves
upper Lipschitz ($l=1$) or H\"older ($l=2$) stable by the
definition, whereas the other part is outside this ball and we can
assume metric regularity. Moreover, we can show that locally
optimal solutions for the perturbed problems $P(\omega)$ exist,
provided $\omega$ is sufficiently close to $\ob$. The obtained
results are partially new even in case of $NLP(\omega)$.

The rest of the paper is organized as follows: In section 2 we
recall the basic definitions of metric (sub)regularity and their
directional versions, together with the characterization of these
properties by objects from generalized differentiation. In section
3 we give some stability results for the feasible point mapping
${\cal F}$. 
Section 4 is devoted to the stability behavior
of solutions of the generalized equation $GE(\omega)$ and the
optimization problem $P(\omega)$, respectively. In section 5 we
apply the obtained results to the special problem $MPEC(\omega)$
by explicitly calculating all objects from generalized
differentiation. Moreover we present some examples.

\section{Preliminaries}
We start by recalling several definitions and results from variational analysis.
Let $\Gamma\subset\R^n$ be an arbitrary closed set and $x\in\Gamma$.
The {\em contingent}  (also called {\em tangent} or {\em Bouligand}) {\em cone} to $\Gamma$ at $x$, denoted by $T(x;\Gamma)$, is
given by
\[T(x;\Gamma):=\{u\in \R^n\mv \exists (u_k)\to u, (t_k)\downarrow 0: x+t_ku_k\in\Gamma\}.\]
We denote by
\begin{equation}\label{DefEpsNormals}
\hat N(x;\Gamma)=\{\xi \in\R^n\mv \limsup_{x'\setto \Gamma
x}\frac{\xi^T( x'-x)}{\norm{x'-x}}\leq 0\}
\end{equation}
the {\em regular} (or {\em Fr\'echet}) {\em normal cone} to $\Gamma$.
Finally, the {\em limiting} (or {\em basic/Mordukhovich}) {\em normal cone} to $\Gamma$ at $x$ is
defined by
\[N(x;\Gamma):=\{\xi \mv \exists  (x_k)\setto{\Gamma}x,\ (\xi_k)\to \xi: \xi_k\in \hat N(x_k;\Gamma) \forall k\}.\]
If $x\not\in \Gamma$, we put $T(x;\Gamma)=\emptyset$, $\hat N(x;\Gamma)=\emptyset$ and $N(x;\Gamma)=\emptyset$.

The limiting normal cone is generally nonconvex, whereas the
regular normal cone is always convex. In the case of a convex
set $\Gamma$, both the regular normal cone and the limiting normal cone
coincide with the standard normal cone from convex analysis and,
moreover, the contingent cone is equal to the tangent cone in the
sense of convex analysis.

Note that $\xi\in\hat N(x;\Gamma)$ $\Leftrightarrow$ $\xi^Tu\leq 0\ \forall u\in T(x;\Gamma)$, i.e., $\hat N(x;\Gamma)$ is the polar cone of $T(x;\Gamma)$.

Given a multifunction $M:\R^n\rightrightarrows \R^m$ and a point $(\xb,\yb)\in \Gr M:=\{(x,y)\in X\times Y\mv y\in M(x)\}$ from its graph, the {\em coderivative} of $M$ at $(\xb,\yb)$ is a multifunction $D^\ast M(\xb,\yb):\R^m\rightrightarrows \R^n$ with the values $D^\ast M(\xb,\yb)(\eta):=\{\xi\in \R^n\mv (\xi,-\eta)\in N((\xb,\yb);\Gr M)\}$, i.e., $D^\ast M(\xb,\yb)(\eta)$ is the collection of all $\xi\in \R^n$ for which there are sequences  $(x_k,y_k)\to (\xb,\yb)$ and $(\xi_k,\eta_k)\to(u,v)$ with $(\xi_k, -\eta_k)\in \hat N((x_k,y_k); \Gr M)$.

For more details we refer to the monographs \cite{Mo06a,RoWe98}

The following directional versions of these limiting constructions were introduced in \cite{Gfr13a}, see also \cite{Gfr14a} for the finite dimensional setting. Given a direction $u\in \R^n$, the limiting normal cone to a subset $\Gamma\subset \R^n$ in direction $u$ at $x\in\Gamma$ is
defined by
\[N(x;\Gamma;u):=\{\xi\in\R^n\mv \exists (t_k)\downarrow 0,\ (u_k)\to u,\ (\xi_k)\to \xi: \xi_k\in \hat N(x+t_ku_k;\Gamma) \forall k\}.\]
For a multifunction $M:\R^n\rightrightarrows \R^m$ and a direction $(u,v)\in \R^n\times \R^m$, the coderivative  of $M$ in direction $(u,v)$ at $(\xb,\yb)\in \Gr M$ is defined as the multifunction $D^\ast M((\xb,\yb);(u,v)):\R^m\rightrightarrows\R^n$ given by $D^\ast M((\xb,\yb);(u,v))(\eta):=\{\xi\in \R^n\mv (\xi,-\eta)\in N((\xb,\yb);\Gr M; (u,v))\}$.

Note that, by the definition, we have $N(x;\Gamma;0)=N(x;\Gamma)$ and $D^\ast M((\xb,\yb);(0,0))=D^\ast M(\xb,\yb)$. Further, $N(x;\Gamma;u)\subset N(x;\Gamma)$ for all $u$ and $N(x;\Gamma;u)=\emptyset$ if $u\not\in T(x;\Gamma)$.

We now turn our attention to the set $P$ from problem \eqref{EqParOptProb}.
For every $y\in P$, we denote by ${\cal P}(y):=\{i\in\{1,\ldots,p\}\mv y\in P_i\}$ the index set of polyhedra containing $y$, and for each $i\in{\cal P}(y)$, we denote by ${\cal A}_i(y):=\{j\in\{1,\ldots,m_i\}\mv a_{ij}^Ty=b_{ij}\}$ the index set of active constraints. Then for every $y\in P$ we have
\[T(y;P)=\bigcup_{i\in{\cal P}(y)} T(y;P_i)=\bigcup_{i\in{\cal P}(y)}\{z\in\R^m\mv a_{ij}^Tz\leq 0, j\in{\cal A}_i(y)\},\]
\[\hat N(y;P)=\bigcap_{i\in{\cal P}(y)}\hat N(y;P_i)=\bigcap_{i\in{\cal P}(y)}\{\sum_{j\in{\cal A}_i(y)}\mu_{ij}a_{ij}\mv \mu_{ij}\geq 0, j\in{\cal A}_i(y)\}.\]
Some formulas for the limiting normal cone respectively its directional counterpart can be found in \cite{Gfr13b}.

The following lemma will be useful for applications.

\begin{lemma}
\label{LemTangDirNormalConeCartProduct}Let $\Gamma=\Gamma_1\times\ldots\times \Gamma_l\subset\R^{m_1}\times\ldots\times\R^{m_l}$ be the Cartesian product of the closed sets $\Gamma_i$ and  $y=(y_1,\ldots,y_l)\in\Gamma$. Then
\begin{equation}\label{EqInclTangCone}
T(y;\Gamma)\subset T(y_1;\Gamma_1)\times\ldots\times T(y_l;\Gamma_l),
\end{equation}
and for every $u=(u_1,\ldots,u_l)\in T(y;\Gamma)$ one has
\begin{equation}\label{EqInclDirNormalCone}
N(y;\Gamma;u)\subset N(y_1;\Gamma_1;u_1)\times\ldots\times N(y_l;\Gamma_l;u_l).
\end{equation}
Furthermore, equality holds in both inclusions if  $\Gamma$  is the union of finitely many convex polyhedra.
\end{lemma}

\begin{proof}The inclusion \eqref{EqInclTangCone} can be found in \cite[Proposition 6.41]{RoWe98}, and the inclusion \eqref{EqInclDirNormalCone} follows immediately from the formula for the regular normal cone from this proposition and the definition of the directional normal cone. To show equality, assume that $\Gamma$ coincides with the set $P$ from \eqref{EqParOptProb},  and let  $(u_1,\ldots,u_l)\in T(y_1;\Gamma_1)\times\ldots\times T(y_l;\Gamma_l)$ and $(\xi_1,\ldots,\xi_l)\in N(y_1;\Gamma_1;u_1)\times\ldots\times N(y_l;\Gamma_l;u_l)$. Then there are sequences $(t_{ik})\downarrow 0$, $(u_{ik})\to u_i$, $(\xi_{ik})\to\xi_i$, $i=1,\ldots,l$, with $z_k:=(y_1+t_{1k}u_{1k},\ldots,y_l+t_{lk}u_{lk})\in\Gamma_1\times\ldots\times\Gamma_l$ and $(\xi_{1k},\ldots,\xi_{lk})\in \hat N(y_1+t_{1k}u_{1k};\Gamma_1)\times\ldots\times\hat N(y_l+t_{lk}u_{lk};\Gamma_l)$ for all $k$. By passing to  subsequences we can assume that there are  index sets ${\cal P}\subset\{1,\ldots,p\}$ and ${\cal A}_i$, $i\in{\cal P}$, such that ${\cal P}={\cal P}(z_k)$ and ${\cal A}_i={\cal A}_i(z_k)$, $i\in{\cal P}$, hold for all $k$. Furthermore, we can assume that for each $i\not\in{\cal P}$ there is an index $j_i$ with $a_{ij_i}z_k>b_{ij_i}$ for all $k$.
Since the convex polyhedra $P_j$ are closed, for each $j\in{\cal P}$ we also have $\lim_k z_k= y\in P_j$ and therefore $(1-\alpha)y +\alpha z_k\in P_j$ $\forall\alpha\in[0,1]$, $\forall k$. Further, for every $\alpha\in (0,1)$ and every $k$ we have ${\cal A}_i={\cal A}_i((1-\alpha)y +\alpha z_k)$ and $a_{ij_i}((1-\alpha)y +\alpha z_k)>b_{ij_i}$, $i\not\in{\cal P}$, showing $\hat N((1-\alpha)y +\alpha z_k;\Gamma)=\hat N(z_k;\Gamma)$ and, together with \cite[Proposition 6.41]{RoWe98},
\begin{eqnarray*}\hat N(z_k;\Gamma)&=&\hat N(y_1+t_{1k}u_{1k};\Gamma_1)\times\ldots\times\hat N(y_l+t_{lk}u_{lk};\Gamma_l)\\
&=&\hat N((1-\alpha)y +\alpha z_k;\Gamma)=\hat N(y_1+\alpha t_{1k}u_{1k};\Gamma_1)\times\ldots\times \hat N(y_l+\alpha t_{lk}u_{lk};\Gamma_l) .
\end{eqnarray*}
 Now let $t_k:=\min_i t_{ik}$. Then for each $i=1,\ldots,l$ we have
\[(1-\frac{t_k}{t_{ik}})y+\frac{t_k}{t_{ik}}z_k=(y_1+t_{1k}\frac{t_k}{t_{ik}}u_{1k},\ldots, y_i+t_ku_{ik},\ldots,y_l+t_{lk}\frac{t_k}{t_{ik}}u_{lk})\in\Gamma_1\times\ldots\times\Gamma_l,\]
\[(\xi_{1k},\ldots,\xi_{lk})\in \hat N(z_k,\Gamma) = \hat N(y_1+t_{1k}\frac{t_k}{t_{ik}}u_{1k};\Gamma_1)\times\ldots\times \hat N(y_i+t_ku_{ik};\Gamma_i)\times\ldots\times\hat N(y_l+t_{lk}\frac{t_k}{t_{ik}}u_{lk};\Gamma_l),\]
showing $y_i+t_ku_{ik}\in\Gamma_i$ and $\xi_{ik}\in\hat N(y_i+t_ku_{ik};\Gamma_i)$.
Hence $\tilde z_k:=(y_1+t_ku_{1k},\ldots,y_l+{t_k}u_{lk})\in \Gamma_1\times\ldots\times\Gamma_l=\Gamma$ and, by using \cite[Proposition 6.41]{RoWe98} again, $(\xi_{1k},\ldots,\xi_{lk})\in \hat N(y_1+t_ku_{1k};\Gamma_1)\times\ldots\times \hat N(y_l+{t_k}u_{lk};\Gamma_l)=\hat N(\tilde z_k;\Gamma)$, showing $(u_1,\ldots,u_l)\in T(x;\Gamma)$ and $(\xi_1,\ldots,\xi_l)\in  N(y;\Gamma;u)$.
\qed\end{proof}

In this paper we are mainly concerned with multifunctions $M:\R^n\rightrightarrows \R^m$ of the form $M(x)=q(x)-\Gamma$, where $q:\R^n\to\R^m$ is smooth and single valued. For this special type of multifunctions, we now give a formula for the directional limiting coderivative of $M$.
\begin{lemma}\label{LemDirCoderiv}
Let the multifunction $M:\R^n\rightrightarrows R^m$ be given by $M(x):=q(x)-\Gamma$, where $q:R^n\to\R^m$ is continuously differentiable and $\Gamma\subset\R^m$ is a closed  set. Further, let $\xb\in q^{-1}(\Gamma)$ and $(u,v)\in\R^n\times\R^m $ be given. Then
\begin{equation}\label{EqCoDeriv}D^\ast M((\xb;0);(u,v))(\lambda)=\{\nabla q(\xb)^T\lambda\mv \lambda\in N(q(\xb);\Gamma; \nabla q(\xb)u-v)\}.
\end{equation}
\end{lemma}
\begin{proof}Let $x^\ast\in D^\ast M((\xb;0);(u,v))(\lambda)$, that is $(x^\ast,-\lambda)\in N((\xb,0);\Gr M; (u,v))$ by the definition, and consider corresponding sequences $(t_k)\downarrow 0$, $(u_k,v_k)\to(u,v)$ and $(x_k^\ast,\lambda_k)\to(x^\ast,\lambda)$ with $(x_k^\ast,-\lambda_k)\in \widehat N((\xb+t_ku_k,t_kv_k);\Gr M)$. Hence $(\xb+t_ku_k,t_kv_k)\in\Gr M$ implying $\gamma_k:=q(\xb+t_ku_k)-t_kv_k\in \Gamma$.
Since for all $x$ and all $\gamma\in\Gamma$
\begin{equation}\label{EqRegNormalConeGrM}\widehat N((x,q(x)-\gamma);\Gr M)=\{(\nabla q(x)^T\eta,-\eta)\mv \eta\in\widehat N(\gamma;\Gamma)\},\end{equation}
we obtain $\lambda_k\in \widehat N(\gamma_k;\Gamma)$ and $x_k^\ast=\nabla q(\xb+t_ku_k)^T\lambda_k$. Hence $x^\ast=\nabla q(\xb)^T\lambda$ and $\lambda\in N(q(\xb);\Gamma;\nabla q(\xb)u-v)$, since
$\lim_{k\to\infty}(\gamma_k-q(\xb))/t_k=\nabla q(\xb)u-v$,
and $D^\ast M((\xb;0);(u,v))(\lambda)\subset\{\nabla q(\xb)^T\lambda\mv \lambda\in N(q(\xb);\Gamma; \nabla q(\xb)u-v)\}$ follows.\\
To show the converse inclusion, we fix an arbitrary element $\lambda\in N(q(\xb);\Gamma; \nabla q(\xb)u-v)$ and consider sequences $(t_k)\downarrow 0$, $(\gamma_k)\to\gamma$ and $(\lambda_k)\to\lambda$ with $\lambda_k\in\widehat N(\gamma_k;\Gamma)$ and $\lim_{k\to\infty}(\gamma_k-q(\xb))/t_k=\nabla q(\xb)u-v$. By setting $u_k:=u$, $v_k:=(q(\xb+t_ku_k)-\gamma_k)/t_k$ we have $(\xb+t_ku_k,t_kv_k)\in\Gr M$ and $(\nabla q(\xb+t_ku_k)^T\lambda_k,-\lambda_k)\in\widehat N((\xb+t_ku_k,q(\xb+t_ku_k)-\gamma_k);\Gr M)=\widehat N((\xb+t_ku_k,t_kv_k);\Gr M)$ by \eqref{EqRegNormalConeGrM}. Passing to the limit and taking into account that
\[\lim_{k\to\infty}v_k=\lim_{k\to\infty}\frac{q(\xb+t_ku_k)-q(\xb)-(\gamma_k-q(\xb))}{t_k}=\nabla q(\xb)u-(\nabla q(\xb)u-v)=v, \]
we obtain $(\nabla q(\xb)^T\lambda,-\lambda)\in N((\xb,0);\Gr M;(u,v))$, and, by the definition of the directional limiting co\-derivative, we can conclude $\nabla q(\xb)^T\lambda\in D^\ast M((\xb;0);(u,v))(\lambda)$. This completes the proof.
\qed\end{proof}

Now we consider the notions of metric regularity and subregularity, respectively, and its characterization by coderivatives and limiting normal cones.

\begin{definition}Let $M:\R^n\rightrightarrows\R^m$ be a multifunction, let $(\xb,\yb)\in\Gr M$ and let $\kappa>0$.
\begin{enumerate}
\item $M$ is called {\em metrically regular with modulus $\kappa$} around $(\xb,\yb)$
if there are neighborhoods $U$ of $\xb$ and $V$ of $\yb$ such that
\begin{equation}
\label{EqMetrReg}
\dist{x,M^{-1}(y)}\leq \kappa\dist{y,M(x)}\ \forall (x,y)\in U\times V.
\end{equation}
\item $M$ is called {\em metrically subregular with modulus $\kappa$} at $(\xb,\yb)$ if there is a neighborhood $U$ of $\xb$ such that
\begin{equation}
\label{EqMetrSubReg}
\dist{x,M^{-1}(\yb)}\leq \kappa\dist{\yb,M(x)}\ \forall x\in U.
\end{equation}
\end{enumerate}
\end{definition}

It is well known that metric regularity of the multifunction $M$ around $(\xb,\yb)$ is equivalent to the Aubin property of the inverse multifunction $M^{-1}$. A multifunction $S:R^m\rightrightarrows R^n$ has the {\em Aubin property} with modulus $L\geq 0$ around some point $(\yb,\xb)\in\Gr S$ if there are neighborhoods $U$ of $\xb$ and $V$ of $\yb$ such that
\[S(y_1)\cap U\subset S(y_2)+L\norm{y_1-y_2}\B_{\R^n}\quad \forall y_1,y_2\in V,\]
where $\B_{\R^n}$ denotes the unit ball in $\R^n$ in the underlying norm.

Metric subregularity of $M$ at $(\xb,\yb)$ is equivalent to the property of {\em calmness} of the inverse multifunction $M^{-1}$. A multifunction $S:R^m\rightrightarrows R^n$ is called {\em calm} with modulus $L\geq 0$ at $(\yb,\xb)\in\Gr S$  if there 
are neighborhoods $U$ of $\xb$ and $V$ of $\yb$ such that
\[S(y)\cap U\subset S(\yb)+L\norm{y-\yb}\B_{\R^n}\quad \forall y\in V.\]

To introduce directional versions of metric (sub)regularity it is convenient to define  the following neighborhoods of a direction: Given   a direction $u\in \R^n$ and positive numbers $\rho,\delta>0$, the set $V_{\rho,\delta}(u)$
 is given by
\begin{equation}
\label{EqDefNbhd}V_{\rho,\delta}(u):=\{z\in\rho \B_{\R^n}\mv
\big\Vert \norm{u} z- \norm{z} u \big\Vert\leq\delta \norm{z}\
\norm{u}\}.
\end{equation}
This can also be written in the form
\[V_{\rho,\delta}(u)=\begin{cases}\{0\}\cup \big\{z\in\rho \B_{\R^n}\setminus \{0\}\mv
\left\Vert \frac z{\norm{z}} - \frac u{\norm{u}} \right\Vert\leq\delta\big\}&\mbox{if $u\not=0$,}\\
\rho \B_{\R^n} &\mbox{if $u=0$.}\end{cases}\]

\begin{definition}Let $M:\R^n\rightrightarrows\R^m$ be a multifunction, let $(\xb,\yb)\in\Gr M$ and let $\kappa>0$, $u\in\R^n$ and $v\in\R^m$.
\begin{enumerate}
\item $M$ is called {\em metrically regular with modulus $\kappa$ in direction $w:=(u,v)$} at $(\xb,\yb)$ if there are positive real numbers $\rho$ and $\delta$  such that
\begin{equation}
\label{EqDirMetrReg}
\dist{x,M^{-1}(y)}\leq \kappa\dist{y,M(x)}
\end{equation}
holds for all $(x,y)\in (\xb,\yb)+V_{\rho,\delta}(w)$ with $\norm{w}\dist{(x,y),\Gr M}\leq \delta\norm{w}\norm{(x,y)-(\xb,\yb)}$.
\item $M$ is called {\em metrically subregular with modulus $\kappa$ in direction $u$} at $(\xb,\yb)$ if there are positive real numbers $\rho$ and $\delta$  such that
\begin{equation}
\label{EqDirMetrSubReg}
\dist{x,M^{-1}(\yb)}\leq \kappa\dist{\yb,M(x)}\ \forall x\in \xb+V_{\rho,\delta}(u).
\end{equation}
\end{enumerate}
\end{definition}
Note that metric regularity in direction $(0,0)$
and metric subregularity in direction $0$ are equivalent to the properties of metric regularity
and metric subregularity, respectively. Further, metric regularity in direction $(u,0)$ implies metric subregularity in direction $u$, cf. \cite[Lemma 1]{Gfr13a}.

\begin{theorem}\label{ThCharMetrSubReg}
Let the multifunction $M:\R^n\rightrightarrows R^m$ be given by $M(x):=q(x)-\Gamma$, where $q:R^n\to\R^m$ is continuously differentiable and $\Gamma\subset\R^m$ is a closed  set. Further let $(\xb,0)\in\Gr M$, $u\in\R^n$ and $v\in\R^m $ be given.
\begin{enumerate}
\item (Mordukhovich criterion): $M$ is metrically regular around $(\xb,0)$ if and only if
\[\nabla q(\xb)^T\lambda=0, \lambda\in N(q(\xb);\Gamma)\ \Longrightarrow \lambda=0.\]
\item $M$ is metrically regular in direction $(u,v)$ at $(\xb,0)$ if and only if
\[\nabla q(\xb)^T\lambda=0, \lambda\in N(q(\xb);\Gamma;\nabla q(\xb)u-v)\ \Longrightarrow \lambda=0.\]
\item Assume that $q$ is twice Fr\'echet differentiable at $\xb$,
that $\Gamma$ is the union of finitely many convex polyhedra and
the condition
\[\nabla q(x)^T\lambda=0,\ \lambda\in N(q(\xb);\Gamma;\nabla q(\xb)u),\ u^T\nabla^2 (\lambda^Tq)(\xb)u\geq0\ \Longrightarrow \lambda=0\]
is fulfilled. Then $M$ is metrically subregular in direction $u$ at $(\xb,0)$.
\end{enumerate}
\end{theorem}

\begin{proof}
The first statement is a specialization of  the more general
statement \cite[Theorem 4.18]{Mo06a}  and can be found e.g. in
\cite[Example 9.44]{RoWe98}. Similarly, the second statement
follows from \cite[Theorem 5]{Gfr13a} by taking into account that
the involved spaces are finite dimensional and \eqref{EqCoDeriv}.
Finally, the last statement follows from  \cite[Theorem
2.6]{Gfr14a}
\qed\end{proof}

Taking into account that in finite dimensions a multifunction is metrically subregular if and only if it is metrically subregular in every nonzero direction, see \cite[Lemma 2.7]{Gfr14a},  and $N(q(\xb);\Gamma;\nabla q(\xb)u)=\emptyset$ if $\nabla q(\xb)u\not\in T(q(\xb);\Gamma)$, we obtain the following consequence of Theorem \ref{ThCharMetrSubReg}:

\begin{corollary}
Let the multifunction $M:\R^n\rightrightarrows R^m$ be  given by
$M(x):=q(x)-\Gamma$, where $q:R^n\to\R^m$ is continuously
differentiable and $\Gamma\subset\R^m$ is a closed  set. Then $M$
is metrically subregular at $(\xb,0)$ if one of the following
conditions is fulfilled:
\begin{enumerate}
\item {\em First order sufficient condition  for metric
subregularity (FOSCMS)}: For every $0\not=u\in\R^n$ with $\nabla
q(\xb)u\in T(q(\xb);\Gamma)$ one has
\[\nabla q(\xb)^T\lambda=0, \lambda\in N(q(\xb);\Gamma;\nabla q(\xb)u)\ \Longrightarrow \lambda=0.\]
\item {\em Second order sufficient condition for metric
subregularity (SOSCMS)}: $q$ is twice Fr\'echet differentiable at
$\xb$, $\Gamma$ is the union of finitely many convex polyhedra, and
for every $0\not=u\in\R^n$ with $\nabla q(\xb)u\in
T(q(\xb);\Gamma)$ one has
\[\nabla q(x)^T\lambda=0,\ \lambda\in N(q(\xb);\Gamma;\nabla q(\xb)u),
\ u^T\nabla^2 (\lambda^Tq)(\xb)u\geq0\ \Longrightarrow \lambda=0.\]
\end{enumerate}
\end{corollary}

Next we consider optimality conditions for the problem
\begin{equation}\label{EqOptProb}\min_x f(x)\quad\mbox{subject to}\quad q(x)\in\Gamma\end{equation}
where $f:\R^n\to\R$, $q:\R^n\to\R^m$ are continuously differentiable and $\Gamma\subset\R^m$ is closed. We denote the feasible region of \eqref{EqOptProb} by ${\cal F}$. Given a feasible point $\xb\in{\cal F}$, we define the {\em linearized cone} by
\[T^{\rm lin}(\xb):=\{u\in\R^n\mv \nabla q(\xb)u\in T(q(\xb);\Gamma)\}\]
and the {\em critical cone} by
\[{\cal C}(\xb)=\{u\in T^{\rm lin}(\xb)\mv \nabla f(\xb)u\leq 0\}.\]

\begin{definition}
Let $\xb\in{\cal F}$ be feasible for the problem \eqref{EqOptProb}.
\begin{enumerate}
\item We say that $\xb$ is {\em B-stationary} if
\[0\in\nabla f(\xb)+\hat N(\xb;{\cal F}).\]
\item We say that $\xb$ is {\em M-stationary} if
\[0\in\nabla f(\xb)+ \nabla q(\xb)^T N(q(\xb);\Gamma).\]
\end{enumerate}
\end{definition}

Since $\hat N(\xb;{\cal F})$ is the polar cone of $T(\xb;{\cal F})$, B-stationarity can be equivalently written as $\nabla f(\xb)u\geq 0$ $\forall u\in T(\xb;{\cal F})$. Hence, B-stationarity means that there does not exist feasible descent directions at $\xb$, which is a first-order necessary condition for $\xb$ being a local minimizer.

Usually B-stationarity is not very useful in practice, since the regular normal cone of ${\cal F}$ at $\xb$ is difficult to compute, in general. Hence, M-stationarity conditions are used as first-order necessary condition, which, however, are only valid under some constraint qualification condition. Indeed,  the following lemma even shows: Under some weak constraint qualification, M-stationarity is not only necessary for local minimizers, but also for B-stationarity.

\begin{lemma}Let $\xb\in{\cal F}$ be B-stationary for the problem \eqref{EqOptProb}, and assume that either ${\cal C}(\xb)=\{0\}$ and the multifunction $\tilde M(u):= \nabla q(\xb)u-T(q(\xb);\Gamma)$ is metrically subregular at $(0,0)$ or there exists $\bar u\in{\cal C}(\xb)$ such that the mapping $M(x):= q(x)-\Gamma$ is metrically subregular in direction $\bar u$ at $(\xb,0)$. Then $\xb$ is M-stationary.
\end{lemma}

\begin{proof}
If ${\cal C}(\xb)=\{0\}$, then $0$ is solution of the problem
\begin{equation}\label{EqLinProbl}\min_{u\in\R^n}\nabla f(\xb) u\quad\mbox{subject to}\quad 0\in\nabla q(\xb)u -T(q(\xb);\Gamma).\end{equation}
Since the mapping $\tilde M(u)$ is assumed to be metrically
subregular at $(0,0)$, by \cite[Corollary 2]{Gfr13a} there is some
$\lambda$ such that $0\in\nabla f(\xb)+D^\ast\tilde
M(0,0)(\lambda)$. By \cite[Example 9.44, Proposition 6.27]{RoWe98} we conclude $-\nabla f(\xb)\in D^\ast\tilde M(0,0)(\lambda)=\{\nabla q(\xb)^T\lambda\}$ and $\lambda\in N(0;T(q(\xb);\Gamma))\subset N(q(\xb);\Gamma)$,  showing M-stationarity of $\xb$.
In the second case, since  $-\nabla f(\xb)\in\hat N(\xb;{\cal F})$ , by \cite[Theorem 6.11]{RoWe98} there is some smooth function $\tilde f:\R^n\to\R$ such that $\nabla \tilde f(\xb)=\nabla f(\xb)$ and $\xb$ is a global minimizer of the problem
\[\min \tilde f(x)\quad\mbox{subject to}\quad x\in{\cal F}=\{x\mv 0\in M(x)\}.\] Now, again by \cite[Corollary 2]{Gfr13a},
we obtain that there is some multiplier $\lambda$ such that $-\nabla f(\xb)=-\nabla \tilde f(\xb)\in D^\ast M(\xb;0)(\lambda)$, and from  \cite[Example 9.44]{RoWe98} we obtain $D^\ast M(\xb;0)(\lambda)=\{\nabla q(\xb)^T\lambda\}$ and $\lambda\in N(q(\xb);\Gamma)$.
\qed\end{proof}

\begin{remark}
If $\Gamma$ is the union of finitely many convex polyhedra, then
so is $T(q(\xb);\Gamma)$, and hence the multifunction $\tilde M$ is
a polyhedral multifunction and consequently metrically subregular
by Robinson's result \cite{Rob81}.
\end{remark}
Given any element  $g\in \hat N(\xb;{\cal F})$, then, by the definition, $\xb$ is a B-stationary solution of the problem
\[\min-g^Tx\quad \mbox{subject to}\quad q(x)\in\Gamma.\]
If $M$ is metrically subregular in $(\xb,0)$, then it is also metrically subregular in every direction, and further, $\tilde M$ is metrically subregular by
\cite[Proposition 2.1]{Gfr11}. Hence $\xb$ is also M-stationary, and we obtain $g\in\nabla q(\xb)^TN(q(\xb);\Gamma)$.
Thus we rediscover the following formula, which would also follow from \cite[Theorem 4.1]{HenJouOut02}:
\begin{corollary}\label{CorInclNormalCone}Let $q:\R^n\to\R^m$ be continuously differentiable, $\Gamma\subset\R^m$ be closed and let $\xb\in q^{-1}(\Gamma)$. If the multifunction $x\rightrightarrows q(x)-\Gamma$ is metrically subregular at $(\xb,0)$, then
\[\hat N(\xb;q^{-1}(\Gamma))\subset\nabla q(\xb)^T N(q(\xb);\Gamma).\]
\end{corollary}

\section{Stability properties of the feasible set mapping ${\cal F}$}
In this section, we study conditions for metric regularity as well as H\"older and Lipschitz stability of the feasible set mapping
${\cal F}$. In particular, we introduce and discuss two regularity properties which imply metric regularity around points \textit{near}
some reference point.

We start with the following technical lemma, where the functions $q$ and $F$ are supposed to satisfy the basic Assumption \ref{AssBasic}:

\begin{lemma}\label{LemTechn}
For every $\epsilon>0$ there are neighborhoods $\hat U_\epsilon$ of $\xb$ and $\hat W_\epsilon$ of $\ob$ such that
\begin{equation}
\label{EqError1}\norm{q(x,\omega)-q(x,\ob)}+\norm{\nabla_x q(x,\omega)-\nabla_x q(x,\ob)}\leq e_1(\omega)+\epsilon\norm{x-\xb},
\end{equation}
\begin{equation}
\label{EqErrorF1}\norm{F(x,\omega)-F(x,\ob)}+\norm{q(x,\omega)-q(x,\ob)}+\norm{\nabla_x q(x,\omega)-\nabla_x q(x,\ob)}\leq \hat \tau_1(\omega)+\epsilon\norm{x-\xb},
\end{equation}
\begin{eqnarray}
\norm{q(x,\omega)-q(x,\ob)}&\leq& \norm{q(\xb,\omega)-q(\xb,\ob)}+e_2(\omega)\norm{x-\xb}+\frac\epsilon 2\norm{x-\xb}^2\nonumber\\
&\leq& e_2(\omega)^2+e_2(\omega)\norm{x-\xb}+\frac\epsilon 2\norm{x-\xb}^2\label{EqError2}
\end{eqnarray}

hold for all $(x,\omega)\in \hat U_{\epsilon}\times \hat W_{\epsilon}$.
\end{lemma}

\begin{proof}Let $\epsilon>0$ be fixed, and choose a ball $\hat U_{\epsilon}$  around $\xb$ and a neighborhood $\hat W_{\epsilon}$ of $\ob$ such that
\[\norm{\nabla_x F(x,\omega)-\nabla_x F(\xb,\ob)}+\norm{\nabla_x q(x,\omega)-\nabla_x q(\xb,\ob)}+\norm{\nabla_x^2 q(x,\omega)-\nabla_x^2 q(\xb,\ob)}\leq\frac\epsilon2,\]
and consequently
\[\norm{\nabla_x F(x,\omega)-\nabla_x F(x,\ob)}+\norm{\nabla_x q(x,\omega)-\nabla_x q(x,\ob)}+\norm{\nabla_x^2 q(x,\omega)-\nabla_x^2 q(x,\ob)}\leq\epsilon\]
hold for all $(x,\omega)\in  \hat U_\epsilon\times \hat W_\epsilon$.

For arbitrarily fixed $(x,\omega)\in \hat U_\epsilon\times \hat W_\epsilon$ we can find $u,\mu\in\B_{\R^n}$, $\xi, \lambda\in\B_{\R^m}$  such that
\begin{eqnarray*}
&&\norm{F(x,\omega)-F(x,\bar\omega)}={\mu}^T(F(x,\omega)-F(x,\bar\omega)),\\
&&\norm{q(x,\omega)-q(x,\bar\omega)}={\xi}^T(q(x,\omega)-q(x,\bar\omega)),\\
&&\norm{\nabla_x q(x,\omega)-\nabla_x q(x,\bar\omega)}={\lambda}^T(\nabla_x q(x,\omega)-\nabla_x q(x,\bar\omega))u.
\end{eqnarray*}
Then there is a point $z$ belonging to the line segment $[\xb,x]\subset U_{\epsilon}$ such that
\begin{eqnarray*}
\lefteqn{{\xi}^T(q(x,\omega)-q(x,\bar\omega))+{\lambda}^T(\nabla_x q(x,\omega)-\nabla_x q(x,\bar\omega))u}\\
&=&{\xi}^T(q(\xb,\omega)-q(\xb,\bar\omega))+{\lambda}^T(\nabla_x q(\xb,\omega)-\nabla_x q(\xb,\bar\omega))u
+\left({\xi}^T(\nabla_x q(z,\omega)-\nabla_x q(z,\bar\omega))\right.\\
&&\quad\left.+{u}^T(\nabla_x^2 ({\lambda}^Tq)(z,\omega)-\nabla_x^2 ({\lambda}^Tq)(z,\bar\omega))\right)(x-\xb)\\
&\leq& e_1(\omega)+\left(\norm{\nabla_x q(z,\omega)-\nabla_x q(z,\bar\omega)}+\norm{\nabla_x^2 q(z,\omega)-\nabla_x^2 q(z,\bar\omega)}\right)\norm{x-\xb},
\end{eqnarray*}
and \eqref{EqError1} follows. The estimate \eqref{EqErrorF1} follows analogously.

To show \eqref{EqError2} note that there is some point $\tilde z$ belonging to the line segment $[\xb,x]$ such that
\begin{eqnarray*}{\xi}^T(q(x,\omega)-q(x,\bar\omega))&=&{\xi}^T(q(\xb,\omega)-q(\xb,\bar\omega))+{\xi}^T(\nabla_x q(\xb,\omega)-\nabla_x q(\xb,\bar\omega))(x-\xb)\\
&&\quad+\frac 12(x-\xb)^T(\nabla_x^2( {\xi}^Tq)(\tilde z,\omega)-\nabla_x^2({\xi}^T q)(\tilde z,\bar\omega))(x-\xb)\\
&\leq&  \norm{q(\xb,\omega)-q(\xb,\bar\omega)}+ e_2(\omega)\norm{x-\xb}\\
&&\qquad+\frac 12\norm{\nabla_x^2 q(\tilde z,\omega)-\nabla_x^2 q(\tilde z,\bar\omega)}\norm{x-\xb}^2,
\end{eqnarray*}
and from this inequality \eqref{EqError2} follows.
\qed\end{proof}

For every $\omega\in\Omega$ we define the linearized cone
\[T^{\rm lin}_\omega(x):=\{u\in\R^n\mv \nabla_x q(x;\omega)u\in T(q(x,\omega);P)\}\] as well as the multifunction $M_{\omega}:\R^n\rightrightarrows \R^m$, $M_{\omega}(x):=q(x,\omega)-P$.
By \cite[Proposition1]{HenOut05} and Corollary \ref{CorInclNormalCone}, respectively, we have
\[T(x;{\cal F}(\omega))=T^{\rm lin}_\omega(x),\quad \hat N(x;{\cal F}(\omega))\subset\nabla_x q(x,\omega)^T N(q(x,\omega);P)\]
for every $x\in {\cal F}(\omega)$ such that $M_\omega$ is metrically subregular at $(x,0)$.

It is well known that the property of metric regularity is stable under small Lipschitzian perturbations, see e.g. \cite[Section 4.1]{KlKum02}, \cite[Section 4.2.3]{Mo06a}. We state here the following result:

\begin{theorem}\label{ThMetrRegulPerturb}
Assume that $M_{\ob}$ is metrically regular around $(\xb,0)$. Then there are neighborhoods $U$ of $\bar x$, $V$ of $0$, $W$ of $\bar\omega$ and a constant $\kappa>0$ such that
\begin{equation}\dist{x,M_\omega^{-1}(y)}\leq \kappa \dist{y,M_\omega(x)}=\kappa \dist{q(x,\omega)-y,P}\quad\forall (x,y,\omega)\in U\times V\times W.\label{EqBndFeasReg}\end{equation}
In particular we have ${\cal F}(\omega)\not=\emptyset$, $\dist{\xb,{\cal F}(\omega)}\leq \kappa\norm{q(\xb,\omega)-q(\xb,\ob)}\leq\kappa e_1(\omega)$ for all $\omega\in W$ and for every $x\in {\cal F}(\omega)\cap U$  with $\omega \in W$  the multifunction $M_\omega$ is metrically regular with modulus $\kappa$ around $(x,0)$.

\end{theorem}
\begin{proof}A similar result was stated in \cite[Lemma
3.1]{GuLiYe12}. However, the given proof appears to be not correct, since the variable $x$ is used in an ambiguous way and therefore, in the notation of \cite{GuLiYe12}, the claimed equation ${\rm dist\,}(x,S_p^{-1}(u)={\rm dist\,}(x, S_{p^\ast}^{-1}(u+F(x,p^\ast)-F(x,p)))$ does not hold true in general. Thus we present a different proof.
Let $\delta>0$,
$\kappa'>0$ be chosen such that
\[
\dist{x,M_{\ob}^{-1}(y)}\leq \kappa' \dist{y,M_{\ob}(x)}\quad \forall (x,y)\in(\xb+\delta \B_{\R^n})\times \delta \B_{\R^m}.
\]
Setting $\epsilon:=\frac 1{24(\kappa'+1)}$,
we denote by $\hat U_\epsilon$ and $\hat W_\epsilon$ the neighborhoods according to Lemma \ref{LemTechn}, and
 we choose some radius $r\in (0,\min\{\frac \delta2,1\})$ with $\xb+r\B_{\R^n}\subset \hat U_\epsilon$. Then we choose some positive radius $\bar r\leq \frac 12 r$ and some neighborhood $W\subset \hat W_\epsilon$ such that
 \[
 \mbox{$\norm{q(x,\omega)-q(\xb,\ob)}< \frac12\epsilon r$ $\forall (x,\omega)\in (\xb+\bar r \B_{\R^n})\times W$
  \ and \ $e_1(\omega)<\epsilon \bar r$ $\forall \omega\in W$.}
 \]
 We now show that the assertion of the theorem holds with $U:=\xb+\bar r\B_{\R^n}$, $V:=\frac 14\epsilon r\B_{\R^m}$ and $\kappa=2(\kappa'+1)$. Consider arbitrarily fixed elements $(\xi,y,\omega)\in U\times V\times W$ and define the functions
 \[
  \mbox{$g'(x):=q(x,\ob)-q(x,\omega)+y$ \ and \  $g(x):=q(x,\ob)-q(x,\omega)+\zeta$,}
 \]
  where $\zeta\in M_\omega(\xi)$ is chosen such that
  \[
  \norm{y-\zeta}=\dist{y,M_\omega(\xi)}=\dist{y, q(\xi,\omega)-P}\leq \norm{y-(q(\xi,\omega)-q(\xb,\ob))}
   \]
 and, consequently, $\norm{\zeta}\leq2\norm{y}+\norm{q(\xi,\omega)-q(\xb,\ob)}<\epsilon r$. Then $g$ and $g'$ are  Lipschitz on $\xb+r\B_{\R^n}$ with constant
less than or equal to
\[\sup\{\norm{\nabla _x q(x,\omega)-\nabla_x q(x,\ob)}\mv x\in \xb+r\B_{\R^n}\}\leq e_1(\omega)+\epsilon r<2\epsilon r\leq 2\epsilon <\frac 1{2(\kappa'+1)},\]
where the first inequality follows from (\ref{EqError1}). Further we  have
\[\sup\{\norm{g'(x)}\mv x\in\xb+r\B_{\R^n}\}< e_1(\omega)+\epsilon r+ \norm{y}<3\epsilon r <\frac r{8(\kappa'+1)}\]
and
\[\sup\{\norm{g(x)}\mv x\in\xb+r\B_{\R^n}\}\leq e_1(\omega)+\epsilon r+ \norm{\zeta}<3\epsilon r =\frac r{8(\kappa'+1)}.\]
Since $g(\xi)\in M_{\ob}(\xi)$, we can apply \cite[Theorem 4.3]{KlKum02} to find some $\xi'$ such that $g'(\xi')\in M_{\ob}(\xi')$ and $\norm{\xi-\xi'}\leq 2(\kappa'+1)\norm{g'(\xi)-g(\xi)}=2(\kappa'+1)\norm{y-\zeta}$. It follows that $y\in q(\xi',\omega)-P$ and thus $\xi'\in M_\omega^{-1}(y)$ and
\[\dist{\xi,M_{\omega}^{-1}(y)}\leq 2(\kappa'+1)\norm{y-\zeta}=2(\kappa'+1)\dist{y,M_{\omega}(\xi)},\]
showing \eqref{EqBndFeasReg}. Taking $\xi=\xb$, $y=0$, we obtain $\dist{\xb,M_{\omega}^{-1}(0)}=\dist{\xb,{\cal F}(\omega)}\leq \kappa \dist{0,M_\omega(\xb)}\leq \kappa\norm{q(\xb,\omega)-q(\xb,\ob)}\leq \kappa e_1(\omega)$ and thus ${\cal F}(\omega)\not=\emptyset$. To complete the proof, note that metric regularity of $M_\omega$ around $(x,0)$, where $x\in{\cal F}(\omega)\cap U$, is a simple consequence of \eqref{EqBndFeasReg}.
\qed\end{proof}

When we do not assume metric regularity of $M_{\bar\omega}$ around $(\xb,0)$, then we cannot expect in general that
the multifunctions $M_\omega$, for $\omega$ near $\bar\omega$, are metrically regular around all points $(x,0)$ with $x\in{\cal F}(\omega)$ close to $\xb$. To handle also this situation, we give the following definition.

\begin{definition}\label{DefTechn}Let  $l\in\{1,2\}$. We say that property $R_l$  holds if there are neighborhoods $U$ of $\bar x$, $W$ of $\bar\omega$ and constants $\kappa>0$ and $\gamma>0$ such that for every $\omega\in W$ and every $x\in {\cal F}(\omega)\cap U$ with $\norm{x-\xb}>\gamma  e_l(\omega)$, the multifunction $M_\omega$ is metrically regular with modulus $\kappa$ around $(x,0)$.
\end{definition}
In particular, properties $R_1$ and $R_2$ imply that $M_{\ob}$ is metrically regular with some uniform modulus around every point $(x,0)$
with $x\in{\cal F}(\ob)\setminus \{\xb\}$ close to $\xb$.

We will now show that property $R_1$ or
property $R_2$ holds if $M_{\ob}$ fulfills at
$(\xb,0)$ the condition FOSCMS or SOSCMS, respectively.

\begin{proposition}\label{PropPropertyR}
\begin{enumerate}
\item Assume that FOSCMS is fulfilled for $M_{\ob}$ at $\xb$, i.e., for every direction $0\not=u\in T^{\rm lin}_{\bar\omega}(\xb)$  we have
\[\nabla_x q(\xb,\bar\omega)^T\lambda=0,\ \lambda\in N(q(\xb,\bar\omega);P;\nabla_x q(\xb,\bar\omega)u) \Longrightarrow \lambda=0.\]
Then property $R_1$ holds.
\item If for every direction $0\not=u\in T^{\rm lin}_{\bar\omega}(\xb)$  we have
\begin{equation}\label{EqSuffCondR2}\nabla_x q(\xb,\bar\omega)^T\lambda=0,\ \lambda\in N(q(\xb,\bar\omega);P;\nabla_x q(\xb,\bar\omega)u), \ u^T\nabla_x^2 (\lambda^Tq)(\bar x,\bar\omega)u=0\ \Longrightarrow \lambda=0,
\end{equation}
then property $R_2$ holds. In particular, SOSCMS implies property $R_2$.
\end{enumerate}
\end{proposition}
\begin{proof}
To prove the first part, assume on the contrary that for every $k$
we can find $(x^k,\omega^k)\in \hat U_{1/k}\times \hat W_{1/k}$
with $x^k\in{\cal F}(\omega^k)\cap (\xb+\frac 1k \B_{R^n})$,
$\norm{\nabla_x q(x^k,\omega^k)-\nabla_x q(\xb,\ob)}\leq\frac 1k$,
$\norm{x^k-\xb}>k e_1(\omega^k)$ such that $M_{\omega^k}$ is not
metrically regular around $(x^k,0)$ with some modulus
less than or equal to
$k$, where $\hat U_{1/k}$, $\hat W_{1/k}$ are as in Lemma
\ref{LemTechn}. By \cite[Example 9.44]{RoWe98}, there is some
$\lambda^k\in N(q(x^k,\omega^k);P)$ with $\norm{\lambda^k}=1$ and
$\norm{\nabla_x q(x^k,\omega^k)^T\lambda^k}\leq \frac 1k$.
According to the definition of the limiting normal cone,
for each $k$ we can find
 elements $q^k\in
(q(x^k,\omega^k)+\frac{\norm{x^k-\xb}}k \B_{\R^n})\cap P$ and
$\xi^k\in\hat N(q^k,P)\cap (\lambda^k+\frac
{\norm{x^k-\xb}}k\B_{\R^n})$. By passing to a subsequence if
necessary, we can assume that $u^k:=(x^k-\xb)/\norm{x^k-\xb}\to u$
and
$\lim_{k\to\infty}\lambda^k=\lim_{k\to\infty}\xi^k=\lambda\not=0$.
Because of Lemma \ref{LemTechn} we have
$\norm{q(x^k,\omega^k)-q(x^k,\ob)}/\norm{x^k-\xb}\leq e_1(\omega^k)/\norm{x^k-\xb} +1/k\leq 2/k$
and therefore
\begin{eqnarray*}
\lim_{k\to\infty}\frac{q^k-q(\xb,\ob)}{\norm{x^k-\xb}}&=&\lim_{k\to\infty}\frac{q(x^k,\omega^k)-q(\xb,\ob)}{\norm{x^k-\xb}}\\
& = &
\lim_{k\to\infty}\left(\frac{q(x^k,\omega^k)-q(x^k,\ob)}{\norm{x^k-\xb}}+\frac{q(x^k,\ob)-q(\xb,\ob)}{\norm{x^k-\xb}}\right)\\
&=&\nabla_xq(\xb,\ob)u,
\end{eqnarray*}
showing $0\not=u\in T^{\rm lin}_{\bar\omega}(\xb)$ and $\lambda\in N(q(\xb,\bar\omega);P;\nabla_x q(\xb,\bar\omega)u)$. Since we also have
\[0=\lim_{k\to\infty}\nabla_x q(x^k,\omega^k)^T\lambda^k=\nabla_x q(\xb,\bar\omega)^T\lambda,\]
we obtain a contradiction to the assumption. Hence the first part is proved.

To prove the second part,  assume on the contrary that for every  $k$ we can
find $(x^k,\omega^k)\in \hat U_{1/k}\times \hat W_{1/k}$ with
$x^k\in{\cal F}(\omega^k)\cap (\xb+\frac 1k \B_{R^n})$,
$\norm{\nabla_x q(x^k,\omega^k)-\nabla_x q(\xb,\ob)}\leq\frac 1k$,
$\norm{x^k-\xb}>k e_2(\omega^k)$ such that $M_{\omega^k}$ is not
metrically regular around $(x^k,0)$ with some modulus
less than or equal to
than $k$.  Then $\norm{q(\xb,\omega^k)-q(\xb,\ob)}^{\frac 12}\leq
e_2(\omega^k)\leq 1/k^2\leq 1$ and  $e_2(\omega^k)\geq
e_1(\omega^k)$ follows.  Hence we can proceed similarly as in the first part of the proof
to find the sequences $(\lambda^k)$, $(\xi^k)$ and $(q^k)$ together
with the limits $0\not=u\in T^{\rm lin}_{\bar\omega}(\xb)$ and
$0\not=\lambda\in N(q(\xb,\bar\omega);P;\nabla_x
q(\xb,\bar\omega)u)$ satisfying $\nabla_x
q(\xb,\bar\omega)^T\lambda=0$, with the only difference that we now
require $q^k\in (q(x^k,\omega^k)+\frac{\norm{x^k-\xb}^2}k
\B_{\R^n})\cap P$. By passing to a subsequence if necessary, we can
assume that there are index sets  ${\cal P}\subset \{1,\ldots,p\}$
and ${\cal A}_i\subset\{1,\ldots,m_i\}$, $i\in{\cal P}$, such that
${\cal P}(q^k)={\cal P}$ and ${\cal A}_i(q^k)={\cal A}_i$,
$i\in{\cal P}$, holds for all $k$. Further  there are numbers
$\mu_{ij}^k\geq 0$, $j\in{\cal A}_i$, $i\in{\cal P}$, such that
\[
 \xi^k-\sum_{j\in{\cal A}_i}\mu_{ij}^ka_{ij}=0, \ i \in{\cal P},
 \]
and $\sum_{j\in{\cal A}_i}\mu_{ij}^k\leq \beta_i\norm{\xi^k}$ for some constant $\beta_i,$
$i \in{\cal P}$.  By passing to a subsequence once more, we can assume that the sequences
$\mu_{ij}^k$ converge to $\mu_{ij}\geq0$ for each $j\in{\cal
A}_i$, $i\in{\cal P}$, and it follows that $\lambda-\sum_{j\in{\cal
A}_i}\mu_{ij}a_{ij}=0,\ i\in{\cal P}$. Since
${\cal P}\subset{\cal P}(q(\xb,\ob))$ and
${\cal A}_i\subset{\cal A}_i(q(\xb,\ob))$, $i\in{\cal P}$,
we obtain for each $i\in{\cal P}$
\[\lambda^T q^k=\sum_{j\in {\cal A}_i}\mu_{ij}a_{ij}^Tq^k=
\sum_{j\in {\cal A}_i}\mu_{ij}b_{ij}=\sum_{j\in {\cal
A}_i}\mu_{ij}a_{ij}^Tq(\xb,\bar\omega)=\lambda^Tq(\xb,\bar\omega)\
\forall k.\]  Using Lemma \ref{LemTechn} we have
\[
\| q(x^k,\omega^k) - q( x^k, \bar \omega) \|
\le e_2(\omega^k)^2 + e_2(\omega^k) \|x^k - \xb\| + \frac{1}{2k} \|x^k - \xb\|^2
\le (\frac{1}{k^2} + \frac{1}{k} + \frac{1}{2k}) \|x^k - \xb\|^2,
\]
and so, together with $\| q^k - q(x^k,\omega^k) \| \le \frac{1}{k} \|x^k - \xb\|^2$ and
$\nabla_x q(\xb,\bar\omega)^T\lambda=0$, we obtain
\begin{eqnarray*}
0&=&\lim_{k\to\infty}\frac{\lambda^T(q^k-q(\xb,\bar\omega))}{\norm{x^k-\xb}^2}=\lim_{k\to\infty}\frac{\lambda^T(q^k-q(x^k,\omega^k) +q(x^k,\omega^k)-q(x^k,\bar\omega)+q(x^k,\bar\omega)-q(\xb,\bar\omega))}{\norm{x^k-\xb}^2}\\
&=&\lim_{k\to\infty}\frac{\lambda^T(q(x^k,\bar\omega)-q(\xb,\bar\omega))}{\norm{x^k-\xb}^2}=\lim_{k\to\infty}\frac{\lambda^T(\nabla_xq(\xb,\bar\omega)(x^k-\xb)+\frac 12 (x^k-\xb)^T\nabla_x^2q(\xb,\bar\omega)(x^k-\xb))}{\norm{x^k-\xb}^2}\\
&=&\frac 12u^T\nabla_x^2(\lambda^T q)(\xb,\bar\omega)u,
\end{eqnarray*}
contradicting
\eqref{EqSuffCondR2}.  The last statement follows from the observation that SOSCMS implies \eqref{EqSuffCondR2}.
\qed\end{proof}

\begin{proposition}\label{PropFeas}Let $\xb\in{\cal F}(\bar\omega)$.
\begin{enumerate}
\item Assume that there is some $0\not=u\in T^{\rm lin}_{\bar\omega}(\xb)$ such that
\[\nabla_x q(\xb,\bar\omega)^T\lambda=0,\ \lambda\in N(q(\xb,\bar\omega);P;\nabla_x q(\xb,\bar\omega)u)\ \Longrightarrow \lambda=0.\]
Then there is a constant $\kappa>0$ and a neighborhood $W$ of $\bar\omega$ such that
\[\dist{\xb,{\cal F}(\omega)}\leq \kappa \norm{q(\xb,\omega)-q(\xb,\ob)}\leq \kappa e_1(\omega)\ \forall \omega\in W.\]
\item Assume that there is some $0\not=u\in T^{\rm lin}_{\bar\omega}(\xb)$ such that
\[\nabla_x q(\xb,\bar\omega)^T\lambda=0,\ \lambda\in N(q(\xb,\bar\omega);P;\nabla_x q(\xb,\bar\omega)u), \ u^T\nabla_x^2 (\lambda^Tq)(\bar x,\bar\omega)u\geq0\ \Longrightarrow \lambda=0.\]
Then there is a constant $\kappa>0$ and a neighborhood $W$ of $\bar\omega$ such that
\[\dist{\xb,{\cal F}(\omega)}\leq \kappa e_2(\omega)\ \forall \omega\in W.\]
\end{enumerate}
\end{proposition}
\begin{proof}We can assume without loss of generality that $\norm{u}=1$. In both cases,  the assumption
 ensures that $M_{\bar\omega}$ is metrically subregular at $(\xb,0)$ in direction
 $u$. Now we make some preliminary considerations, before proving
 the assertions of the proposition.
For every $k$ we can find a neighborhood $W^k$ of $\bar\omega$ and
a radius $\rho^k>0$ such that $(\xb+\rho^k\B_{\R^n})\times
W^k\subset\hat U_{1/k}\times \hat W_{1/k}$,
\begin{eqnarray*}
&&\sup\{\norm{\nabla_x^2 q(x,\omega)-\nabla_x^2 q(x,\bar\omega)}\mv x\in\xb+\rho^k\B_{\R^n},\omega\in W^k\}<\frac 1{9k^3},\\
&&\sup\{\norm{\nabla_x q(x,\omega)-\nabla_x q(\xb,\bar\omega)}\mv x\in\xb+\rho^k\B_{\R^n},\omega\in W^k\}<\frac 1k,\\
&&\sup\{e_2(\omega)\mv \omega\in W^k\}\leq \min\{\frac 1{16k^3},\frac{\rho^k}{16k^2}\}.
\end{eqnarray*}
 Now consider  sequences $(t^k)\downarrow 0$ and $(\omega^k)$ such that $t^k<\rho^k/2$ and $\omega^k\in W^k$ hold for all $k$. Since $\nabla_x q(\xb,\bar\omega)u\in T(q(\xb,\bar\omega);P)$ and $P$ is the union of finitely many convex polyhedral sets, we have $q(\xb,\bar\omega)+t^k\nabla_x q(\xb,\bar\omega)u\in P$ for all $k$ sufficiently large. Thus there is some constant $L$ such that $\dist{q(\xb +t^ku,\bar\omega),P}\leq L(t^k)^2$ and, by metric subregularity of $M_{\bar\omega}$ in direction $u$, there is some $\kappa'>0$ such that for every $k$ there is some point $\xb^k\in{\cal F}(\bar\omega)\cap (\xb+t^ku+\kappa'L(t^k)^2\B_{\R^n})$. For all $k$ sufficiently large we also have $t^kL\kappa'<\frac 12$, implying
\[\frac{t^k}2<\norm{\xb^k-\xb}<\frac{3t^k}2<\rho^k,\] and we obtain
\begin{eqnarray*}\dist{q(\xb^k,\omega^k),P}&\leq& \dist{q(\xb^k,\bar\omega),P}+\norm{q(\xb^k,\omega^k)-q(\xb^k,\bar\omega)}\\
&\leq& 0+\norm{q(\xb,\omega^k)-q(\xb,\bar\omega))}+
\norm{\xb^k-\xb}\norm{\nabla_x q(\xb,\omega^k)-\nabla_x q(\xb,\bar\omega)}\\
&&\qquad+\frac{\norm{\xb^k-\xb}^2}2\sup\{\norm{\nabla_x^2 q(x,\omega^k)-\nabla_x^2 q(x,\bar\omega)}\mv \norm{x-\xb}\leq\norm{\xb^k-\xb}\}\\
&\leq& \norm{q(\xb,\omega^k)-q(\xb,\bar\omega))}+
\frac 32t^k\norm{\nabla_x q(\xb,\omega^k)-\nabla_x q(\xb,\bar\omega)}+ \frac{(t^k)^2}{8k^3}=:\epsilon^k.
\end{eqnarray*}
By using Ekeland's variational principle, we can find for every $k$ some point $x^k\in \xb^k+k\epsilon^k\B_{\R^n}$ such that $\dist{q(x^k,\omega^k),P}\leq \dist{q(\xb^k,\omega^k),P}$ and
\[\dist{q(x^k,\omega^k),P}\leq \dist{q(x,\omega^k),P}+\frac 1k \norm{x- x^k} \ \forall x\in \R^n.\]
Since $P$ is closed, there is for every $k$ some $q^k\in P$ with $\dist{q(x^k,\omega^k),P}=\norm{q(x^k,\omega^k)-q^k}$, and we conclude
\[\norm{q(x^k,\omega^k)-q^k}\leq \norm{q(x,\omega^k)-y}+\frac 1k \norm{x-x^k} \ \forall (x,y)\in \R^n\times P,\]
i.e., $(x^k,q^k)$ is a solution of the optimization problem
\[\min_{x,y} \norm{q(x,\omega^k)-y}+\frac 1k \norm{x-x^k} \quad\mbox{subject to}\quad y\in P.\]
By applying first-order optimality conditions, it follows that
$\xi^k:=(q( x^k,\omega^k)-q^k)/\norm{q( x^k,\omega^k)-q^k}\in\hat N(q^k;P)$ and $\norm{\nabla_x q(x^k,\omega^k)^T\xi^k}\leq \frac 1k$, provided that $q(x^k,\omega^k)-q^k\not=0$ .

Now let us prove the first assertion by contraposition. We assume
on the contrary that for every $k$ there is some $\omega^k\in W^k$
with  $\dist{\xb,{\cal
F}(\omega^k)}>12k^2\norm{q(\xb,\omega^k)-q(\xb,\ob)}$. Note that
$\norm{q(\xb,\omega^k)-q(\xb,\ob)}=0$ or $e_2(\omega^k)=0$ implies
$\xb\in{\cal F}(\omega^k)$ and thus
$\norm{q(\xb,\omega^k)-q(\xb,\ob)}>0$. By taking
$t^k:=4k^2\norm{q(\xb,\omega^k)-q(\xb,\ob)}\leq 4k^2
e_1(\omega^k)\leq 4k^2e_2(\omega^k)\leq\min\{\frac 1{4k}, \frac
{\rho^k}4\}$  we obtain
\begin{eqnarray*}\epsilon^k&\leq&
\norm{q(\xb,\omega^k)-q(\xb,\ob)}(1+6k^2 \norm{\nabla_x
q(\xb,\omega^k)-\nabla_x q(\xb,\bar\omega)}+\frac{t^k}{2k})\\
&\leq&
\norm{q(\xb,\omega^k)-q(\xb,\ob)}(1+\frac3{8k}+\frac{t^k}{2k})\leq
2\norm{q(\xb,\omega^k)-q(\xb,\ob)}= \frac1{2k^2}t^k<
\frac{\norm{\xb^k-\xb}}{k^2}
\end{eqnarray*}
 and therefore
$\norm{x^k-\xb^k}\leq k\epsilon^k< \frac{\norm{\xb^k-\xb}}{k}$,
showing
\begin{equation}\label{EqProofProp2Eq1}
\lim_{k\to\infty}\frac{x^k-\xb}{\norm{x^k-\xb}}=\lim_{k\to\infty}\frac{\xb^k-\xb}{\norm{\xb^k-\xb}}=u,
\end{equation}
\begin{eqnarray}\nonumber(1-\frac 1k)\norm{\xb^k-\xb}&<&\norm{x^k-\xb}< (1+\frac
1k)\norm{\xb^k-\xb}\\\label{EqProofProp2Eq2} &<&\frac 32(1+\frac
1k)t^k\leq12k^2 \norm{q(\xb,\omega^k)-q(\xb,\ob)}<\dist{\xb,{\cal
F}(\omega^k)},
\end{eqnarray}
and
\begin{equation}\label{EqProofProp2Eq3}\norm{q(x^k,\omega^k)-q^k}\leq\epsilon^k<
\frac{\norm{\xb^k-\xb}}{k^2}\leq \frac{\norm{x^k-\xb}}k\ \forall
k\geq 2.
\end{equation}
 From \eqref{EqProofProp2Eq2} we can conclude
$q(x^k,\omega^k)-q^k\not=0$, since otherwise $\dist{\xb,{\cal
F}(\omega^k)}\leq\norm{\xb-x^k}$ would hold. Hence  $\xi^k$ is
well defined, and we can proceed as in the proof of Proposition
\ref{PropPropertyR} to obtain the contradiction that every limit
point $\lambda$ of the sequence $(\xi^k)$ fulfills
$\norm{\lambda}=1$, $\lambda\in N(q(\xb);P;\nabla_xq(\xb,\ob)u)$
and $\nabla_x q(\xb,\ob)^T\lambda=0$.

We prove the second part similarly. Assuming on the contrary that
for every $k$ there is some $\omega^k\in W^k$ with
$e_2(\omega^k)>0$  and $\dist{\xb,{\cal
F}(\omega^k)}/e_2(\omega^k)>24k^2$ and setting $t^k:=8k^2
e_2(\omega^k)\leq\min\{\frac1{2k},\frac {\rho^k}2\}$, we obtain
$\norm{\xb^k-\xb}<\frac 32 t^k\leq \frac 34$ and, together with
$\frac 12 t^k<\norm{\xb^k-\xb}$,
\[\epsilon^k\leq e_2(\omega^k)^2(1+12k^2+8k)=
(t^k)^2\frac{1+12k^2+8k}{64k^4}<\frac{\norm{\xb^k-\xb}^2}{k^2}\leq\frac{\norm{\xb^k-\xb}}{k^2}\
\forall k\geq 2.\] As before  we obtain $\norm{x^k-\xb^k}\leq
k\epsilon^k< \frac{\norm{\xb^k-\xb}}{k}$, and thus
\eqref{EqProofProp2Eq1} as well as
\[(1-\frac
1k)\norm{\xb^k-\xb}<\norm{x^k-\xb}< (1+\frac
1k)\norm{\xb^k-\xb}<\frac 32(1+\frac 1k)t^k\leq24k^2
e_2(\omega^k)<\dist{\xb,{\cal F}(\omega^k)}\] and
\eqref{EqProofProp2Eq3} hold. This implies again that $(\xi^k)$ is
well defined; using the same arguments as in the proof of the
second part of Proposition \ref{PropPropertyR}, we obtain that
every limit point $\lambda$ of the sequence $(\xi^k)$ fulfills
$\norm{\lambda}=1$, $\lambda\in N(q(\xb);P;\nabla_xq(\xb,\ob)u)$,
$\nabla_x q(\xb,\ob)^T\lambda=0$ and
$u^T\nabla_x^2(\lambda^Tq)(\xb,\ob)u=0$, a contradiction.
\qed\end{proof}

\section{Lipschitz and H\"older stability  of solution mappings}
In what follows we denote by $X:\Omega\rightrightarrows\R^n$ the mapping which assigns to every $\omega$ the set of local minimizers for the problem $P(\omega)$ in (\ref{EqParOptProb}),
and by $S:\Omega\rightrightarrows \R^n$ and $S_M:\Omega\rightrightarrows\R^n$
the mappings which assign to every $\omega\in\Omega$ the sets of B-stationary points and M-stationary points, respectively, for the problem $P(\omega)$, i.e.,
\[S(\omega):=\{x\in\R^n\mv 0\in \nabla_x f(x,\omega)+\hat N(x;{\cal F}(\omega))\},\]
\[S_M(\omega):=\{x\in\R^n\mv 0\in \nabla_x f(x,\omega)+\nabla_x q(x,\omega)^T N(q(x,\omega);P)\}.\]
Then we have $X(\omega)\subset S(\omega)$ $\forall\omega\in\Omega$.

Further we consider the solution mapping $\hat S:\Omega\rightrightarrows\R^n$ of the generalized equation \eqref{EqGE},
\[\hat S(\omega):=\{x\in\R^n\mv 0\in  F(x,\omega)+\hat N(x;{\cal F}(\omega))\},\]
and the related mapping
\[\hat S_M(\omega):=\{x\in\R^n\mv 0\in F(x,\omega)+\nabla_x q(x,\omega)^T N(q(x,\omega);P)\}.\]

The first result of this section are  sufficient conditions such that estimates of the form
\[(S(\omega)\cup S_M(\omega))\cap U\subset \xb+L\tau_l(\omega)\B_{\R^n}\ \forall \omega\in W\]
or
\[(\hat S(\omega)\cup \hat S_M(\omega))\cap U\subset \xb+L\hat\tau_l(\omega)\B_{\R^n}\ \forall \omega\in W\]
hold, where $L>0$ and $U$ and $W$ are neighborhoods of $\xb$ and $\ob$. 
The following result is only stated  for the solution mappings $\hat S(\omega)\cup \hat S_M(\omega)$ for $GE(\omega)$, the corresponding statement for $P(\omega)$ follows immediately by taking $F(x,\omega):=\nabla_x f(x,\omega)$.

\begin{theorem}\label{ThStab}
Let $\xb\in {\cal F}(\bar\omega)$.
\begin{enumerate}
\item If property $R_1$ holds and there does not exist a triple $(u,\lambda,\mu)\in \R^n\times\R^m\times\R^m$ satisfying
\begin{eqnarray}
&&0\not=u\in T^{\rm lin}_{\bar\omega}(\xb),\label{EqLipCond_u}\\
&&\lambda\in N(q(\xb,\bar\omega);P;\nabla_x q(\xb,\bar\omega)u)\label{EqLipCond_lambda}\\
&&\mu\in T(\lambda;N(q(\xb,\bar\omega);P;\nabla_x q(\xb,\bar\omega)u))\label{EqLipCond_mu}\\
&&F(\xb,\bar\omega)+\nabla_x q(\xb,\bar\omega)^T\lambda=0\label{EqLipCond_FOEq}\\
&&\nabla_x F(\xb,\bar\omega)u +\nabla_x^2(\lambda^Tq)(\xb,\bar\omega)u+\nabla_x q(\xb,\bar\omega)^T\mu=0,\label{EqLipCond_SOEq}
\end{eqnarray}
then there are neighborhoods $U$ of $\xb$, $W$ of $\bar\omega$ and a constant $L>0$ such that
\[(\hat S(\omega)\cup \hat S_M(\omega))\cap U\subset \xb+L\hat\tau_1(\omega)\B_{\R^n}\ \forall\omega\in W.\]

\item If property $R_2$ holds and there does not exist a quadruple $(u,\lambda,\mu,v)\in \R^n\times\R^m\times\R^m\times \R^n$ such that $(u,\lambda,\mu)$ fulfills \eqref{EqLipCond_u}-\eqref{EqLipCond_SOEq} and
\begin{eqnarray}
&&\nabla_xq(\xb,\bar\omega)v+u^T\nabla_x^2q(\xb,\bar\omega)u\in T(q(\xb,\bar\omega);\bigcup_{i\in{\cal I}(u)}P_i),\label{EqHoeldCond1}\\
&&u^T\nabla_x^2(\lambda^Tq)(\xb,\bar\omega)u=F(\xb,\bar\omega)^Tv,\label{EqHoeldCond2},\end{eqnarray}
where ${\cal I}(u):=\{i\mv \nabla_xq(\xb,\bar\omega)u\in T(q(\xb,\bar\omega);P_i)\}$,
then there are neighborhoods $U$ of $\xb$, $W$ of $\bar\omega$ and a constant $L>0$ such that
\[(\hat S(\omega)\cup \hat S_M(\omega))\cap U\subset \xb+L\hat\tau_2(\omega)\B_{\R^n}\ \forall\omega\in W.\]
\end{enumerate}
\end{theorem}

\begin{proof}
We prove the first part by contraposition. Let $W$ denote the neighborhood according to the definition of property $R_1$. Assume on the contrary that for every $k$ we can find $(x^k,\omega^k)\in(\hat U_{1/k}\cap (\xb+\frac 1k\B_{\R^n}))\times (\hat W_{1/k}\cap W)$ with  $(x^k)\in \hat S_M(\omega^k)\cup \hat S(\omega^k)$
and $\norm{x^k-\xb}>k\hat\tau_1(\omega^k)$. Because of $e_1(\omega^k)\leq\hat \tau_1(\omega^k)$ it follows that $e_1(\omega^k)/\norm{x^k-\xb}\leq\hat \tau_1(\omega^k)/\norm{x^k-\xb}< \frac1{k}$, and,  due to property $R_1$, we have that $M_{\omega^k}$ is metrically regular with modulus $\kappa$ around $(x^k,0)$ for all $k$ sufficiently large. Hence, $x^k\in \hat S_M(\omega^k)$ and there exists a vector $\lambda^k\in N(q(x^k,\omega^k);P)$ with $F(x^k,\omega^k)+\nabla_x q(x^k,\omega^k)^T\lambda^k=0$. From \cite[Example 9.44]{RoWe98} we conclude $\norm{\lambda^k}\leq \kappa\norm{F(x^k,\omega^k)}$, showing that $(\lambda^k)$ is bounded. By the definition of the limiting normal cone, we can find for each $k$ elements $q^k\in (q(x^k,\omega^k)+\frac{\norm{x^k-\xb}^2}k \B_{\R^n})\cap P$ and $\xi^k\in\hat N(q^k,P)\cap (\lambda^k+\frac {\norm{x^k-\xb}}k\B_{\R^n})$. By passing to a subsequence if necessary, we can assume that $u^k:=(x^k-\xb)/\norm{x^k-\xb}\to u$ and $\lambda^k\to\lambda$.
Because of  Lemma \ref{LemTechn} we have $\norm{q(x^k,\omega^k)-q(x^k,\ob)}/\norm{x^k-\xb}\leq \hat \tau_1(\omega^k)/\norm{x^k-\xb} +1/k\to 0$ and therefore
\begin{eqnarray*}
\lim_{k\to\infty}\frac{q^k-q(\xb,\ob)}{\norm{x^k-\xb}}&=&\lim_{k\to\infty}\frac{q(x^k,\omega^k)-q(\xb,\ob)}{\norm{x^k-\xb}}=
\lim_{k\to\infty} \frac{q(x^k,\ob)-q(\xb,\ob)}{\norm{x^k-\xb}}
=\nabla_xq(\xb,\ob)u,
\end{eqnarray*}
showing $0\not=u\in T^{\rm lin}_{\bar\omega}(\xb)$ and $\lambda\in N(q(\xb,\bar\omega);P;\nabla_x q(\xb,\bar\omega)u)$. Using Lemma \ref{LemTechn} again, we similarly obtain
\begin{equation}\label{EqLimit1}\lim_{k\to\infty}\frac{\nabla_x q(x^k,\omega^k)-\nabla_x q(\xb,\ob)}{\norm{x^k-\xb}}=\nabla_x^2q(\xb,\ob)u,\quad \lim_{k\to\infty}\frac{F(x^k,\omega^k)-F(\xb,\ob)}{\norm{x^k-\xb}}=\nabla_x F(\xb,\ob)u.
\end{equation}
Further
\begin{equation}\label{EqLimit1a}
0=\lim_{k\to\infty}(F(x^k,\omega^k)+\nabla_x q(x^k,\omega^k)^T\lambda^k)=F(\xb,\bar\omega)+\nabla_x q(\xb,\bar\omega)^T\lambda.
\end{equation}
By passing to a subsequence once more, we can assume that there are
index sets ${\cal P}\subset \{1,\ldots,p\}$ and ${\cal
A}_i\subset\{1,\ldots,m_i\}$, $i\in{\cal P}$, such that ${\cal
P}(q^k)={\cal P}$ and ${\cal A}_i(q^k)={\cal A}_i$, $i\in{\cal P}$,
holds for all $k$. Further, by the generalized Farkas lemma, cf. \cite[Proposition 2.201]{BonSh00},  there are numbers $\zeta_{ij}^k\geq 0$,
$j\in{\cal A}_i$, $i\in{\cal P}$, such that
\[\xi^k-\sum_{j\in{\cal A}_i}\zeta_{ij}^ka_{ij}=0,\ i\in{\cal P},\]
and $\sum_{j\in{\cal A}_i}\zeta_{ij}^k\leq \beta_i\norm{\xi^k}$,
$i\in{\cal P}$, for some constants $\beta_i$. By passing to a
subsequence once more, we can assume that the sequences
$\zeta_{ij}^k$ converge to $\zeta_{ij}\geq0$ for each $j\in{\cal
A}_i$, $i\in{\cal P}$, and it follows that $\lambda-\sum_{j\in{\cal
A}_i}\zeta_{ij}a_{ij}=0,\ i\in{\cal P}$. Further,
$F(\xb,\bar\omega)+\nabla_x q(\xb,\bar\omega)^T\lambda=0$ and
thus, by Hoffman's lemma, there is some real $\bar\gamma>0$ such
that  for each $k$ we can find $\bar \xi^k\in\R^m$ and nonnegative
numbers $\bar\zeta_{ij}^k$, $j\in{\cal A}_i$, $i\in{\cal P}$,
satisfying
\begin{eqnarray*}
&&\norm{\bar \xi^k-\xi^k}+\sum_{i\in{\cal P}}\sum_{j\in{\cal A}_i}\vert\zeta^k_{ij}-\bar\zeta_{ij}^k\vert\leq\bar\gamma \norm{F(\xb,\bar\omega)+\nabla_x q(\xb,\bar\omega)^T\xi^k},\\
&&\bar \xi^k-\sum_{j\in{\cal A}_i}\bar\zeta_{ij}^ka_{ij}=0,\ i\in{\cal P},\\
&&F(\xb,\bar\omega)+\nabla_x q(\xb,\bar\omega)^T\bar\xi^k=0.
\end{eqnarray*}
Taking into account $F(x^k,\omega^k)+\nabla_x q(x^k,\omega^k)^T\lambda^k=0$ and
$\|\nabla_x q(\xb,\bar\omega)^T(\lambda^k-\xi^k)\| \le(\norm{\nabla_x q(\xb,\bar\omega)}\norm{x_k-\xb})/k$, we obtain
\begin{eqnarray*}\lefteqn{\limsup_{k\to\infty}\frac{\norm{F(\xb,\bar\omega)+\nabla_x q(\xb,\bar\omega)^T\xi^k}}{\norm{x^k-\xb}}}\\
&=&
\limsup_{k\to\infty}\frac{\norm{F(\xb,\bar\omega)+\nabla_x q(\xb,\bar\omega)^T\xi^k-F(x^k,\omega^k)-\nabla_x q(x^k,\omega^k)^T\lambda^k}}{\norm{x^k-\xb}}\\
&\leq&
\limsup_{k\to\infty} \frac{\norm{F(\xb,\bar\omega)-F(x^k,\omega^k)+(\nabla_x q(\xb,\bar\omega)-\nabla_x q(x^k,\omega^k))^T\lambda^k}
+ \norm{\nabla_x q(\xb,\bar\omega)^T(\lambda^k-\xi^k)}}{\norm{x^k-\xb}} \\
&=&\norm{\nabla_x F(\xb,\bar\omega)u+\nabla_x^2(\lambda^Tq)(\xb,\bar\omega)u}<\infty,
\end{eqnarray*}
where the last equation follows from \eqref{EqLimit1}.
Hence for each $j\in{\cal A}_i$, $i\in{\cal P}$, the sequence $(\zeta_{ij}^k-\bar\zeta_{ij}^k)/\norm{x^k-\xb}$ is bounded, and we can assume that it converges to some $\nu_{ij}$, where we have eventually passed to a subsequence. Now consider the set ${\cal I}:=\{(i,j)\mv  i\in{\cal P}, j\in{\cal A}_i, \zeta_{ij}=0\}$. If $\nu_{ij}\geq0$ $\forall (i,j)\in{\cal I}$ we set $\bar \nu_{ij}=\nu_{ij}$, $j\in{\cal A}_i$, $i\in{\cal P}$. Otherwise we choose an index $\bar k$ such that
$(\zeta_{ij}^{\bar k}-\bar\zeta_{ij}^{\bar k})/\norm{x^{\bar k}-\xb}<\nu_{ij}/2$ hold for all $(i,j)\in{\cal I}$ with $\nu_{ij}<0$ and set $\bar \nu_{ij}=\nu_{ij}+2(\bar\zeta_{ij}^{\bar k}-\zeta_{ij})/\norm{x^{\bar k}-\xb}$, $j\in{\cal A}_i$, $i\in{\cal P}$. Then we have for every $(i,j)\in{\cal I}$,
\[\bar \nu_{ij}=\nu_{ij}+2\bar\zeta_{ij}^{\bar k}/\norm{x^{\bar k}-\xb}\geq \nu_{ij}+2(\bar\zeta_{ij}^{\bar k}-\zeta_{ij}^{\bar k})/\norm{x^{\bar k}-\xb}>0,\]
and hence in any case there is some $\bar t>0$ such that $\zeta_{ij}+t\bar\nu_{ij}\geq0$ holds for all $j\in{\cal A}_i$, $i\in{\cal P}$ and $t\in[0,\bar t]$.
Setting $\mu:=\sum_{j\in{\cal A}_i}\bar\nu_{ij}a_{ij}$ for an arbitrarily chosen $i\in {\cal P}$, we either have $\mu=\lim_{k\to\infty}\frac{\xi^k-\bar\xi^k}{\norm{x^k-\xb}}$ or
$\mu=\lim_{k\to\infty}\frac{\xi^k-\bar\xi^k}{\norm{x^k-\xb}}+2\frac{\bar\xi^{\bar k}-\lambda}{\norm{x^{\bar k}-\xb}}$, and therefore $\mu=\sum_{j\in{\cal A}_i}\bar\nu_{ij}a_{ij}$ holds for all $i\in{\cal P}$. Hence $\lambda+t\mu\in\hat N(q^k;P)$ $\forall k$ and $\lambda+t\mu\in N(q(\xb,\bar\omega);P;\nabla_x q(\xb,\bar\omega)u)$, and $\mu\in T(\lambda;N(q(\xb,\bar\omega);P;\nabla_x q(\xb,\bar\omega)u))$ follows. Taking into account $\nabla_x q(\xb,\bar\omega)^T\lambda=\nabla_x q(\xb,\bar\omega)^T\bar\xi^{\bar k}=-F(\xb,\bar\omega)$, we obtain
\begin{eqnarray*}
0&=&\lim_{k\to\infty}\frac{F(x^k,\omega^k)+\nabla_x q(x^k,\omega^k)^T\lambda^k}{\norm{x^k-\xb}}=\lim_{k\to\infty}\frac{F(x^k,\omega^k)+\nabla_x q(x^k,\omega^k)^T\xi^k}{\norm{x^k-\xb}}\\
&=&\lim_{k\to\infty}\left(\frac{F(x^k,\omega^k)+\nabla_x q(x^k,\omega^k)^T\bar\xi^k}{\norm{x^k-\xb}}+\nabla_x q(x^k,\omega^k)^T\frac{\xi^k-\bar\xi^k}{\norm{x^k-\xb}}\right)\\
&=&\nabla_x q(\xb,\bar\omega)^T\mu+\lim_{k\to\infty}\frac{F(x^k,\omega^k)-F(\xb,\bar\omega)+(\nabla_x q(x^k,\omega^k)-\nabla_x q(\xb,\bar\omega))^T\bar\xi^k}{\norm{x^k-\xb}}\\
&=&\nabla_x q(\xb,\bar\omega)^T\mu+\nabla_x F(\xb,\bar\omega)u+\nabla_x^2(\lambda^Tq)(\xb,\bar\omega)u.
\end{eqnarray*}
Thus the triple $(u,\lambda,\mu)$ fulfills conditions \eqref{EqLipCond_u}-\eqref{EqLipCond_SOEq}, a contradiction, and the first part is proved.

We also prove the second  assertion by contraposition. Let $W$ now denote the neighborhood according to the definition of property $R_2$. Assume on the contrary that for every $k$ we can find $(x^k,\omega^k)\in(\hat U_{1/k}\cap (\xb+\frac 1k\B_{\R^n}))\times (\hat W_{1/k}\cap W)$ with  $(x^k)\in \hat S_M(\omega^k)\cup \hat S(\omega^k)$
and $\norm{x^k-\xb}>k\hat\tau_2(\omega^k)$. Then $\hat \tau_2(\omega^k)\leq 1/k^2$ and consequently $\hat\tau_1(\omega^k)/\norm{x^k-\xb}\leq \hat\tau_2(\omega^k)/\norm{x^k-\xb}\leq 1/k$ for all $k$,
and we can proceed as in the first part to find $(u,\lambda,\mu)$. Thus, in order to prove the second assertion, it remains to show that there is some $v$ such that \eqref{EqHoeldCond1}-\eqref{EqHoeldCond2} holds. Since ${\cal P}\subset{\cal P}(q(\xb,\bar\omega))$ and ${\cal A}_i\subset{\cal A}_i(q(\xb,\bar\omega))$, $i\in{\cal P}$, we have
$\lambda^T(q^k-q(\xb,\bar\omega))=(\sum_{j\in{\cal A}_i}\zeta_{ij}a_{ij})^T(q^k-q(\xb,\ob))=\sum_{j\in{\cal A}_i}\zeta_{ij}(b_{ij}-b_{ij})=0$ for all $k$. Now fix any $\bar i\in{\cal P}$ and consider an arbitrarily fixed vector $\xi\in\R^m$ such that $F(\xb,\bar\omega)+\nabla_x q(\xb,\bar\omega)^T\xi=0$ and there exist nonnegative numbers $\tau_{\bar ij}\geq 0$, $j\in{\cal A}_{\bar i}(q(\xb,\bar\omega))$ with $\xi=\sum_{j\in{\cal A}_{\bar i}(q(\xb,\bar\omega))}\tau_{\bar ij}a_{\bar ij}$.
Since $q^k\in P_{\bar i}$ we have $a_{\bar ij}^Tq^k\leq b_{\bar ij}=a_{\bar ij}^Tq(\xb,\bar\omega)$, $j\in{\cal A}_{\bar i}(q(\xb,\bar\omega))$ and therefore $\xi^T(q^k-q(\xb,\bar\omega))\leq 0$. Hence,
\[
\begin{array}{rcl}
0 &  \geq & (\xi-\lambda)^T(q^k-q(\xb,\bar\omega))\\
& = & (\xi-\lambda)^T \left( q^k-q(x^k,\omega^k) + q(x^k,\omega^k)-q(x^k,\bar\omega)\right.\\
& & \qquad \qquad \quad + \left. \nabla_x q(\xb,\bar\omega)^T (x^k-\xb)+ \frac 12(x^k-\xb)^T\nabla_x^2q(\xb,\bar\omega)(x^k-\xb)\right)+r^k,
\end{array}
\]
where $r^k:=(\xi-\lambda)^T\left(q(x^k,\bar\omega)-q(\xb,\bar\omega)-\nabla_x q(\xb,\bar\omega)(x^k-\xb)- \frac 12(x^k-\xb)^T\nabla_x^2q(\xb,\bar\omega)(x^k-\xb)\right)$. By Lemma \ref{LemTechn} we have $(q(x^k,\omega^k)-q(x^k,\bar\omega))/\norm{x^k-\xb}^2\to 0$. Since $(\xi-\lambda)^T\nabla_xq(\xb,\bar\omega)=0$, $(q^k-q(x^k,\omega^k))/\norm{x^k-\xb}^2\to 0$ and $r^k/\norm{x^k-\xb}^2\to 0$, by dividing by $\norm{x^k-\xb}^2$ and taking the limit $k\to\infty$,  we obtain
\[0\geq \frac 12u^T\nabla_x^2((\xi-\lambda)^Tq)(\xb,\bar\omega)u.\]
Setting $\zeta_{\bar ij}=0$, $j\in{\cal A}_{\bar i}(q(\xb,\bar\omega))\setminus {\cal A}_{\bar i}$, we obtain that $\zeta_{\bar ij}$, $j\in {\cal A}_{\bar i}(q(\xb,\bar\omega))$, is a solution of the linear optimization problem
\begin{eqnarray*}
\min \sum_{j\in{\cal A}_{\bar i}(q(\xb,\bar\omega))}-\tau_{\bar ij}a_{\bar ij}^T(u^T\nabla_x^2q(\xb,\bar\omega)u)&&\\
\mbox{subject to } \sum_{j\in {\cal A}_{\bar i}(q(\xb,\bar\omega))}\tau_{\bar ij}\nabla_xq(\xb,\bar\omega)^Ta_{\bar ij}&=&-F(\xb,\bar\omega)\\
\tau_{\bar ij}&\geq& 0,\ j\in{\cal A}_{\bar i}(q(\xb,\bar\omega)).
\end{eqnarray*}
Then, by duality theory of linear optimization, the dual program also has a solution $v\in\R^n$ and
\begin{eqnarray*}
-F(\xb,\bar\omega)^T v&=&-u^T\nabla_x^2(\lambda^Tq)(\xb,\bar\omega)u\\
a_{\bar ij}^T\nabla_x q(\xb,\bar\omega)v&\leq& -a_{\bar ij}^T(u^T\nabla_x^2 q(\xb,\bar\omega)u),\ j\in{\cal A}_{\bar i}(q(\xb,\bar\omega)).
\end{eqnarray*}
Hence $\nabla_x q(\xb,\bar\omega)v+u^T\nabla_x^2 q(\xb,\bar\omega)u\in T(q(\xb,\bar\omega),P_{\bar i})$, and since $\bar i\in{\cal P}\subset {\cal I}(u)$, the quadruple $(u,\lambda,\mu, v)$ fulfills 
\eqref{EqLipCond_u}-\eqref{EqHoeldCond2}, a contradiction, and the second part is also proved.
\qed\end{proof}

Now we consider the solution mappings for the the problem \eqref{EqParOptProb}. It turns out that the assumptions of Theorem \ref{ThStab} can be partially replaced by a second-order sufficient condition and a quadratic growth condition, respectively. Further we can guarantee the existence of locally optimal solutions.

Recall that a point $\xb \in \cal F(\bar\omega)$ is an {\em essential local minimizer of second order} for the problem $P(\bar\omega)$ if there is a neighborhood $U$ of $\xb$ and a constant $\eta>0$ such that
\begin{equation}\label{EqEssLocMin}\max\{f(x,\bar\omega)-f(\xb,\bar\omega),\ \dist{q(x,\bar\omega),P}\}\geq \eta\norm{x-\xb}^2\ \forall x\in U.\end{equation}
This implies that the {\em quadratic growth condition}  for the problem $P(\bar\omega)$ holds at $\xb$, i.e.,
$f(x,\bar\omega)-f(\xb,\bar\omega) \geq \eta\norm{x-\xb}^2\ \forall x\in U \cap {\cal F}(\bar\omega)$.
The opposite direction is true if $M_{\bar\omega}$ is metrically subregular at $(\xb,\bar\omega)$.
To see this one could use similar arguments as in \cite[Section 3]{Gfr06} by noting that convexity of $P$
is not needed  and the assumption of metric regularity used in \cite{Gfr06} can be replaced by assuming metric subregularity.

Note that the proof of the following proposition does not use the polyhedral form of $P$, it suffices to suppose that $P$ is a nonempty closed set.
\begin{proposition}\label{PropExistLocMin}
Let $\xb$ be an essential local minimizer of second order for the problem $P(\bar\omega)$. Then there are constants $\gamma_1,\gamma_2>0$ and a neighborhood $W$ of $\ob$ such that for all $\omega\in W$ with
$\dist{\xb,{\cal F}(\omega)}~<~\gamma_1$ one has
\[\dist{\xb,X(\omega)}\leq \gamma_2\left(\dist{\xb,{\cal F}(\omega)}^{\frac 12}+\tau_2(\omega)\right)\]
\end{proposition}
\begin{proof}
Let $\eta>0$ and $\rho>0$ be such that \eqref{EqEssLocMin} hold for all $x\in\xb+\rho\B_{\R^n}$, and define $L:=\sup\{\norm{\nabla_xf(x,\ob)}\mv x\in \xb+\rho\B_{\R^n}\}+1$, $\gamma_2:=\max\{\frac {1+\sqrt{2\eta+1}}\eta, 2\sqrt{L/\eta}\}$. Then we choose $0<\bar \rho\leq\rho$ and $W\subset \hat W_{\eta/2}$ such that $\xb+\bar\rho\B_{\R^n}\subset \hat U_{\eta/2}$,
\[\norm{\nabla_x^2f(x,\omega)-\nabla_x^2f(\xb,\ob)}\leq\frac\eta4\ \forall(x,\omega)\in(\xb+\bar\rho\R_{\B^n})\times W,\]
and finally, that $\tau_2(\omega)< \min\{L/2, \bar\rho/(2\gamma_2)\}$ $\forall \omega\in W$. Further, we set 
\[\gamma_1:=\min\{2L/\eta,\bar\rho^2/(4\gamma_2^2)\}\leq \min\{2L/\eta,\bar\rho^2/(16 L/\eta)\}\leq\bar\rho.\]
Now let $\omega\in W$ with $\dist{\xb,{\cal F}(\omega)}<\gamma_1$ be arbitrarily fixed, and choose $y\in{\cal F}(\omega)$ with $\norm{y-\xb}=\dist{\xb,{\cal F}(\omega)}$. For $x\in\xb+\bar\rho\B_{\R^n}$ we obtain
\begin{eqnarray*}
&&\vert f(x,\omega)-f(x,\ob)-(f(y,\omega)-f(y,\ob))- (\nabla_x f(\xb,\omega)-\nabla_x f(\xb,\ob))(x-y)\vert\\
&&\leq\frac{\eta}4(\norm{x-\xb}^2+\norm{y-\xb}^2)
\end{eqnarray*}
and, together with Lemma \ref{LemTechn},
\begin{eqnarray*}
&&\vert f(x,\omega)-f(x,\ob)-(f(y,\omega)-f(\xb,\ob))\vert+\norm{q(x,\omega)-q(x,\ob)}\\
&&\leq e_2(\omega)^2+\tau_2(\omega)(\norm{x-\xb}+\norm{y-\xb})+\frac\eta2\norm{x-\xb}^2+\frac{\eta}4\norm{y-\xb}^2+\vert f(y,\ob)-f(\xb,\ob)\vert\\
&&\leq \tau_2(\omega)^2+\tau_2(\omega)\norm{x-\xb} + \frac\eta2\norm{x-\xb}^2 +(\tau_2(\omega)+\frac \eta4 \norm{y-\xb}+L)\norm{y-\xb}\\
&&\leq \tau_2(\omega)^2+\tau_2(\omega)\norm{x-\xb} + \frac\eta2\norm{x-\xb}^2 +2L\norm{y-\xb},
\end{eqnarray*}
implying
\begin{eqnarray*}
\alpha(x,y,\omega)&:=&\max\{f(x,\omega)-f(y,\omega), \dist{q(x,\omega),P}\}\\
&\geq& \max\{f(x,\ob)-f(\xb,\ob), \dist{q(x,\ob),P}\}\\
&&\qquad-\vert f(x,\omega)-f(x,\ob)-(f(y,\omega)-f(\xb,\ob))\vert-\norm{q(x,\omega)-q(x,\ob)}\\
&\geq& \frac\eta2\norm{x-\xb}^2-\tau_2(\omega)^2-\tau_2(\omega)\norm{x-\xb} - 2L\norm{y-\xb}.
\end{eqnarray*}
Now assume $\norm{x-\xb}>\gamma_2(\norm{y-\xb}^{\frac12}+\tau_2(\omega))$. Then
\[
\norm{x-\xb}>2\sqrt{\frac L\eta}\norm{y-\xb}^{\frac12}+\frac {1+\sqrt{2\eta+1}}{\eta}\tau_2(\omega),
\]
and therefore
\begin{eqnarray*}
\frac{\eta}2\left(\norm{x-\xb}-\frac{\tau_2(\omega)}{\eta}\right)^2&>&\frac{\eta}2\left(2\sqrt{\frac L\eta}(\norm{y-\xb}^{\frac12}+\frac {\sqrt{2\eta+1}}{\eta}\tau_2(\omega)\right)^2\\
&\geq&2L\norm{y-\xb}+\frac{2\eta+1}{2\eta}\tau_2(\omega)^2,
\end{eqnarray*}
showing $\alpha(x,y,\omega)>0$. Hence we conclude that for every $x\in{\cal F}(\omega)$ with $\gamma_2(\norm{y-\xb}^{\frac12}+\tau_2(\omega))<\norm{x-\xb}\leq\bar \rho$ we have $f(x,\omega)>f(y,\omega)$, showing, together with $\gamma_2(\norm{y-\xb}^{\frac12}+\tau_2(\omega))<\bar \rho$, that there is a global solution $\xb_\omega$ of the problem
\[\min\ f(x,\omega)\ \mbox{subject to}\ q(x,\omega)\in P,\ x\in\xb+\bar\rho B_{\R^n},\]
with $\xb_\omega\in X(\omega)$ and $\norm{\xb_\omega-\xb}\leq \gamma_2(\norm{y-\xb}^{\frac12}+\tau_2(\omega))<\bar\rho$.
\qed\end{proof}
Given $x\in{\cal F}(\omega)$, we denote by ${\cal C}_{\omega}(x):=\{u\in T_{\omega}^{\rm lin}(x)\mv \nabla_x f(x,\omega)u\leq0\}$ the cone of critical directions at $x$. Further we introduce the Lagrangian ${\cal L}:\R^n\times\R^m\times\Omega\to\R$,
\begin{equation}\label{EqLagr}{\cal L}(x,\lambda,\omega):=f(x,\omega)+\lambda^Tq(x,\omega),
\end{equation}
and for every $\omega\in\Omega$, every $x\in{\cal F}(\omega)$ and every $u\in{\cal C}_\omega(x)$ we define the set of multipliers
\[\Lambda^1_\omega(x;u):=\{\lambda\in N(q(x,\omega);P;\nabla_xq(x,\omega)u)\mv \nabla_x{\cal L}(x,\lambda,\omega)=0\}\]
\begin{definition}Let $\xb\in{\cal F}(\bar\omega)$. We say that the {\em refined strong second-order sufficient condition} (RSSOSC) holds at $\xb$ for the problem $P(\bar\omega)$ if for every nonzero critical direction $0\not=u\in{\cal C}_{\bar\omega}(\xb)$   one has
\[u^T\nabla_x^2 {\cal L}(\xb,\lambda,\bar\omega)u>0\ \ \forall\lambda\in \Lambda^1_{\bar \omega}(\xb;u).\]
\end{definition}
Note that RSSOSC is sufficient for $\xb$ being an essential local minimizer of second order only under some additional first-order optimality condition, e.g., if $\xb$ is an extended M-stationary point in the sense of \cite{Gfr14a}, i.e. $\Lambda^1_{\bar\omega}(\xb;u)\not=\emptyset$ $\forall 0\not=u\in{\cal C}_{\bar\omega}(\xb)$.

Recall that the multifunction $M_\omega$ is defined by $M_\omega(x) = q(x,\omega) -P$.
In what follows, the first- and second order sufficient conditions for metric subregularity FOSCMS and SOSCMS are used as in Proposition \ref{PropPropertyR}.

\begin{theorem}\label{ThStabLocMin}
Let $\xb\in {\cal F}(\bar\omega)$, and assume that Assumption \ref{AssBasic} is fulfilled.
\begin{enumerate}
\item If FOSCMS is fulfilled for $M_{\ob}$ at $\xb$ and
RSSOSC holds at $\xb$ for $P(\bar\omega)$, then there are neighborhoods $U$ of $\xb$, $W$ of $\bar\omega$ and a constant $L>0$ such that
\begin{equation}\label{EqUppLischOpt}(S_M(\omega)\cup S(\omega))\cap U\subset \xb+L\tau_1(\omega)\B_{\R^n}\ \forall\omega\in W.\end{equation}
In addition, if $\xb$ is an essential local minimizer for $P(\ob)$ and either $T^{\rm lin}_{\ob}(\xb)\not=\{0\}$ or $M_{\ob}$ is metrically regular around $(\xb,0)$, then $X(\omega)\cap U\not=\emptyset$ $\forall \omega\in W$.
\item If SOSCMS is fulfilled for $M_{\ob}$ at $\xb$ and $\xb$ is an essential local minimizer of second order for $P(\bar\omega)$,
then there are neighborhoods $U$ of $\xb$, $W$ of $\bar\omega$ and a constant $L>0$ such that
\begin{equation}\label{EqUppHoeldOpt}(S_M(\omega)\cup S(\omega))\cap U\subset\xb+L\tau_2(\omega)\B_{\R^n}\ \forall\omega\in W.\end{equation}
In addition, if either $T^{\rm lin}_{\ob}(\xb)\not=\{0\}$  or $M_{\ob}$ is metrically regular around $(\xb,0)$, then $X(\omega)\cap U\not=\emptyset$ $\forall \omega\in W$.
\end{enumerate}
\end{theorem}

\begin{proof}
We show the first part by contraposition. If the assertion does not hold, by virtue of Theorem \ref{ThStab} with $F=\nabla_x f$ together with Proposition \ref{PropPropertyR}, there is some triple $(u,\lambda,\mu)$ satisfying $\lambda\in\Lambda^1_{\bar\omega}(\xb;u)$, \eqref{EqLipCond_u}, \eqref{EqLipCond_mu}  and $0=\nabla_x^2{\cal L}(\xb,\lambda,\bar\omega)u+\nabla_x q(\xb;\bar\omega)^T\mu$. By \eqref{EqLipCond_mu},  there are sequences $(\alpha^k)\downarrow 0$ and $(\mu^k)\to\mu$ with $\lambda+\alpha^k\mu^k\in N(q(\xb,\bar\omega);P;\nabla_xq(\xb,\bar\omega)u)$, and by using \cite[Lemma 2.1]{Gfr14a} we obtain $\lambda^T\nabla_x q(\xb,\bar\omega)u=(\lambda+\alpha^k\mu^k)^T\nabla_x q(\xb,\bar\omega)u=0$ $\forall k$, showing $\lambda^T\nabla_x q(\xb,\bar\omega)u=\mu^T\nabla_x q(\xb,\bar\omega)u=0$. Hence $u\in{\cal C}_{\bar\omega}(\xb)$ because of $0=\nabla_x{\cal L}(\xb,\lambda,\bar\omega)u=\nabla_xf(\xb,\bar\omega)u$. It follows
$0=u^T\nabla_x^2{\cal L}(\xb,\lambda,\bar\omega)u$, contradicting RSSOSC.

We also prove the second part by contraposition. Assuming on the contrary that the assertion does not hold, by applying Theorem \ref{ThStab} with $F=\nabla_x f$ together with Proposition \ref{PropPropertyR}, we find $(u,\lambda,\mu,v)$ satisfying $\lambda\in\Lambda^1_{\bar\omega}(\xb;u)$, \eqref{EqLipCond_u},  \eqref{EqLipCond_mu}  and $0=\nabla_x^2{\cal L}(\xb,\lambda,\bar\omega)u+\nabla_x q(\xb;\bar\omega)^T\mu$, \eqref{EqHoeldCond1} and $u^T\nabla_x^2(\lambda^Tq)(\xb,\bar\omega)u=\nabla_x f(\xb,\bar\omega)^Tv$. Proceeding as before, we obtain $u\in{\cal C}_{\bar\omega}(\xb)$ and $0=u^T\nabla_x^2{\cal L}(\xb,\lambda,\bar\omega)u=u^T\nabla_x^2 f(\xb,\bar\omega)u+\nabla_x f(\xb,\bar\omega)^Tv\in T(\nabla_x f(\xb,\bar\omega)u;\R_-)$. Further, we have $T(q(\xb,\bar\omega); P_i)\subset T(\nabla_xq(\xb,\bar\omega)u; T(q(\xb,\bar\omega); P_i))$ for each $i\in{\cal I}(u)$ and thus
\begin{eqnarray*}\nabla_xq(\xb,\bar\omega)v+u^T\nabla_x^2q(\xb,\bar\omega)u&\in& T(q(\xb,\bar\omega);\bigcup_{i\in{\cal I}(u)}P_i)=\bigcup_{i\in{\cal I}(u)}T(q(\xb,\bar\omega);P_i)\\
&\subset& \bigcup_{i\in{\cal I}(u)}T(\nabla_xq(\xb,\bar\omega)u; T(q(\xb,\bar\omega); P_i))\\
&=&T(\nabla_xq(\xb,\bar\omega)u; T(q(\xb,\bar\omega);\bigcup_{i\in{\cal I}(u)} P_i))\\
&=&T(\nabla_xq(\xb,\bar\omega)u; T(q(\xb,\bar\omega);P)).
\end{eqnarray*}
Hence we can conclude from \cite[Lemma 3.16]{Gfr14a} that $\xb$ is not an essential local minimizer of second order for $P(\bar\omega)$ contradicting our assumption.

Finally, in both cases the assertion that $X(\omega)\cap U\not=\emptyset$ $\forall\omega\in W$ follows from Proposition \ref{PropFeas} respectively Theorem \ref{ThMetrRegulPerturb} and Proposition \ref{PropExistLocMin}.
\qed\end{proof}

\begin{remark}The statements \eqref{EqUppLischOpt} and \eqref{EqUppHoeldOpt}, respectively, remain valid if we replace FOSCMS and SOSCMS by the formally weaker assumption that properties $R_1$ and $R_2$, respectively, hold.  In order that the statement $X(\omega)\cap U\not=\emptyset$ holds in case that $\xb$ is an essential local minimizer, we must ensure that ${\cal F}(\omega)\not=\emptyset$ for $\omega$ close to $\ob$, e.g. by the assumption that there is a direction $0\not=u\in T^{\rm lin}_{\ob}(\xb)$ fulfilling the assumptions of Proposition \ref{PropFeas} or by the assumption of metric regularity of $M_\ob$ around $(\xb,0)$.
\end{remark}

\begin{remark}\label{RemMetrReg}
If $p=1$, i.e. if $P$ is a polyhedron, then it follows from \cite[Proposition 3.9]{Gfr11} that the conditions FOSCMS for $M_{\ob}$ and $T^{\rm lin}_{\ob}(\xb)\not=\{0\}$ imply metric regularity of $M_{\ob}$ around $(\xb,0)$.
\end{remark}

\section{Application to MPECs}

In this section we consider the special case $MPEC(\omega)$ as given in Example \ref{ExMPEC}.  In what follows we denote by $\xb$ a  point feasible for the problem $MPEC(\ob)$. Recall that the set $P$ is given by $P=\R^{m_I}_-\times\{0\}^{m_E}\times Q_{EC}^{m_C}$. Hence, by \cite[Proposition 6.41]{RoWe98} and by Lemma \ref{LemTangDirNormalConeCartProduct}, for every $(\tilde g, \tilde h, \tilde a)\in P\subset\R^{m_I}\times\R^{m_E}\times(\R^2)^{m_C}$ we have
\begin{eqnarray*}
T((\tilde g, \tilde h, \tilde a);P)&=&T(\tilde g;\R^{m_I}_-)\times \{0\}^{m_E}\times \prod_{i=1}^{m_C}T(\tilde a_i;Q_{EC}),\\
N((\tilde g, \tilde h, \tilde a);P)&=&N(\tilde g;\R^{m_I}_-)\times \R^{m_E}\times \prod_{i=1}^{m_C}N(\tilde a_i;Q_{EC}),\\
\hat N((\tilde g, \tilde h, \tilde a);P)&=&\hat N(\tilde g;\R^{m_I}_-)\times \R^{m_E}\times \prod_{i=1}^{m_C}\hat N(\tilde a_i;Q_{EC}).
\end{eqnarray*}
Further,  for every direction $(u,v,w)\in T((\tilde g, \tilde h, \tilde a);P)$ we have
\[N((\tilde g, \tilde h, \tilde a);P;(u,v,w))= N(\tilde g;\R^{m_I}_-;u)\times \R^{m_E}\times \prod_{i=1}^{m_C}N(\tilde a_i;Q_{EC};w_i).\]
For the inequality constraints we obviously have
\begin{eqnarray*}T(\tilde g;\R^{m_I}_-)&=&\{u\in\R^{m_I}\mv u_i\leq 0\  \forall i:\tilde g_i=0\},\\
 N(\tilde g;\R^{m_I}_-)&=&\hat N(\tilde g;\R^{m_I}_-)=\{\lambda\in\R^{m_I}_+\mv \lambda_i\tilde g_i=0,\ i=1,\ldots,m_I\},\\
 N(\tilde g;\R^{m_I}_-;u)&=&\{\lambda\in\hat N(\tilde g;\R^{m_I}_-)\mv \lambda_iu_i=0,\ i=1,\ldots,m_I\}.
\end{eqnarray*}

 By straightforward calculation we can obtain the formulas for the regular normal
cone, the limiting normal cone and the contingent cone of the set $Q_{\rm EC}$ as follows:
 For all $a=(a_1,a_2)\in Q_{\rm EC}$ we have
\[\hat N(a;Q_{\rm EC})=\left\{(\xi_1,\xi_2)\mv\begin{array}{ll} \xi_2=0&\mbox{if $0=a_1>a_2$}\\
\xi_1\geq0,\xi_2\geq 0&\mbox{if $a_1=a_2=0$}\\
\xi_1=0&\mbox{if $a_1<a_2=0$}
\end{array}\right\},\]
\[N(a;Q_{\rm EC})=\begin{cases}\hat N(a;Q_{\rm EC})&\mbox{if $a\not=(0,0)$}\\
\{(\xi_1,\xi_2)\mv\mbox{either $\xi_1>0$, $\xi_2>0$ or $\xi_1\xi_2=0$}\}&\mbox{if $a=(0,0)$},
\end{cases}\]
\[T(a;Q_{\rm EC})=\left\{(u_1,u_2)\mv\begin{array}{ll}u_1=0&\mbox{if $0=a_1>a_2$}\\
u_1\leq 0,u_2\leq0, u_1u_2=0&\mbox{if $a_1=a_2=0$}\\
u_2=0&\mbox{if $a_1<a_2=0$}
\end{array}\right\},\]
and for all $u=(u_1,u_2)\in T(a;Q_{\rm EC})$ we have
\[T(u;T(a;Q_{\rm EC}))=\begin{cases}T(a;Q_{\rm EC})&\mbox{if $a\not=(0,0)$}\\
T(u;Q_{\rm EC}))&\mbox{if $a=(0,0)$},
\end{cases}\]
\[\hat N(u;T(a;Q_{\rm EC}))=\begin{cases}\hat N(a;Q_{\rm EC})&\mbox{if $a\not=(0,0)$}\\
\hat N(u;Q_{\rm EC})&\mbox{if $a=(0,0)$},
\end{cases}\]
\[N(a;Q_{\rm EC};u)=\begin{cases}N(a;Q_{\rm EC})&\mbox{if $a\not=(0,0)$}\\
 N(u;Q_{\rm EC})&\mbox{if $a=(0,0)$}.
\end{cases}\]

Denoting
\begin{eqnarray*}
&&\bar I_g:=\{i\in\{1,\ldots,m_I\}\mv g_i(\xb,\ob)=0\},\\
&&\bar I^{+0}:=\{i\in\{1,\ldots,m_C\}\mv G_i(\xb,\ob)>0=H_i(\xb,\ob)\},\\
&&\bar I^{0+}:=\{i\in\{1,\ldots,m_C\}\mv G_i(\xb,\ob)=0<H_i(\xb,\ob)\},\\
&&\bar I^{00}:=\{i\in\{1,\ldots,m_C\}\mv G_i(\xb,\ob)=0=H_i(\xb,\ob)\},
\end{eqnarray*}
the cone $T_\ob^{\rm lin}(\xb)$ is given by
\[T_\ob^{\rm lin}(\xb)=\left\{u\in\R^n\mv \begin{array}{l}\nabla_x g_i(\xb,\ob)u\leq 0,\ i\in\bar I_g,\\
\nabla_x h_i(\xb,\ob)u=0,\ i=1,\ldots,m_E,\\
\nabla_x G_i(\xb,\ob)u=0,\ i\in \bar I^{0+},\\
\nabla_x H_i(\xb,\ob)u=0,\ i\in \bar I^{+0},\\
-(\nabla_x G_i(\xb,\ob)u,\nabla_x H_i(\xb,\ob)u)\in Q_{\rm EC},\ i\in \bar I^{00}
\end{array}\right\}.\]
Given $u\in T^{\rm lin}_\ob(\xb)$ we define
\begin{eqnarray*}&&I_g(u):=\{i\in\bar I_g\mv\nabla_x g_i(\xb,\ob)u=0\}\\
&&I^{+0}(u):=\{i\in \bar I^{00}\mv \nabla_x G_i(\xb,\ob)u>0=\nabla_x H_i(\xb,\ob)u\},\\
&&I^{0+}(u):=\{i\in \bar I^{00}\mv \nabla_x G_i(\xb,\ob)u=0<\nabla_x H_i(\xb,\ob)u\},\\
&&I^{00}(u):=\{i\in \bar I^{00}\mv \nabla_x G_i(\xb,\ob)u=0=\nabla_x H_i(\xb,\ob)u\}.
\end{eqnarray*}
For given $u\in T_\ob^{\rm lin}(\xb)$, the set $N(q(\xb,\ob);P;\nabla_x q(\xb,\ob)u)$ consists of all $\lambda=(\lambda^g,\lambda^h,\lambda^G,\lambda^H)\in\R^{m_I}\times \R^{m_E}\times \R^{m_C}\times\R^{m_C}$ satisfying
\begin{eqnarray}
\label{EqFOSCMS1}&&\lambda^g_i\geq 0,\ i\in I_g(u), \quad \lambda^g_i= 0,\ i\not\in I_g(u),\\
\label{EqFOSCMS2}&&\lambda^H_i=0,\ i\in \bar I^{0+}\cup I^{0+}(u),\quad\lambda^G_i=0,\ i\in \bar I^{+0}\cup I^{+0}(u),\\
\label{EqFOSCMS3}&&\mbox{either $\lambda^G_i>0,\lambda^H_i>0$ or $\lambda^G_i\lambda^H_i=0$},\ i\in I^{00}(u)
\end{eqnarray}

Then we have the following characterizations of metric subregularity: FOSCMS is fulfilled for $M_\ob$ at $\xb$ if and only if for every direction $0\not=u\in T_\ob^{\rm lin}(\xb)$ the only multiplier $\lambda=(\lambda^g,\lambda^h,\lambda^G,\lambda^H)\in\R^{m_I}\times \R^{m_E}\times \R^{m_C}\times\R^{m_C}$ satisfying
\eqref{EqFOSCMS1}-\eqref{EqFOSCMS3} and
\begin{equation}
\label{EqFOSCMS4}\sum_{i=1}^{m_I}\lambda_i^g\nabla_x g_i(\xb,\ob)+\sum_{i=1}^{m_E}\lambda_i^h\nabla_x h_i(\xb,\ob)-\sum_{i=1}^{m_C}(\lambda_i^G\nabla_x G_i(\xb,\ob)+\lambda_i^H\nabla_x H_i(\xb,\ob))=0
\end{equation}
is the trivial multiplier $\lambda=0$.

Similarly,  SOSCMS is fulfilled for $M_\ob$ at $\xb$ if and only if for every direction $0\not=u\in T_\ob^{\rm lin}(\xb)$ the only multiplier $\lambda=(\lambda^g,\lambda^h,\lambda^G,\lambda^H)\in\R^{m_I}\times \R^{m_E}\times \R^{m_C}\times\R^{m_C}$ satisfying \eqref{EqFOSCMS1}-\eqref{EqFOSCMS4}  and
\begin{equation}
\label{EqSOSCMS}u^T\left(\sum_{i=1}^{m_I}\lambda_i^g\nabla_x^2 g_i(\xb,\ob)+\sum_{i=1}^{m_E}\lambda_i^h\nabla_x^2 h_i(\xb,\ob)-\sum_{i=1}^{m_C}(\lambda_i^G\nabla_x^2 G_i(\xb,\ob)+\lambda_i^H\nabla^2_x H_i(\xb,\ob))\right)u\geq 0
\end{equation}
is $\lambda=0$.

Further, $M_\ob$ is metrically regular around $(\xb,0)$ if and only if in case $u=0$ the only multiplier $\lambda$ fulfilling \eqref{EqFOSCMS1}-\eqref{EqFOSCMS4}
is $\lambda=0$. Note that $ I_g(0)=\bar I_g$, $I^{+0}(0)=I^{0+}(0)=\emptyset$ and $I^{00}(0)=\bar I^{00}$.

We now translate Theorem \ref{ThStab} into terms of $MPEC(\omega)$. To do this it is convenient to consider the (extended) Lagrangian
\[{\cal L}^{\lambda_0}(x,\lambda,\omega):=\lambda_0 f(x,\omega)+\sum_{i=1}^{m_I}\lambda_i^g g_i(\xb,\omega)+\sum_{i=1}^{m_E}\lambda_i^h h_i(\xb,\omega)-\sum_{i=1}^{m_C}(\lambda_i^G G_i(\xb,\omega)+\lambda_i^H H_i(\xb,\omega))\]
where $\lambda_0\in\R$, $x\in\R^n$, $\lambda=(\lambda^g,\lambda^h,\lambda^G,\lambda^H)\in\R^{m_I}\times \R^{m_E}\times \R^{m_C}\times\R^{m_C}$ and $\omega\in\Omega$. The Lagrangian ${\cal L}^1$ corresponds to the Lagrangian ${\cal L}$ as defined in \eqref{EqLagr}.

\begin{corollary}\label{CorStabMPEC}
Let $\xb$ be feasible for $MPEC(\ob)$.
\begin{enumerate}
\item If property $R_1$ holds and there does not exist $0\not=u\in T^{\rm lin}_\ob(\xb)$, $\lambda=(\lambda^g,\lambda^h,\lambda^G,\lambda^H)$ and $\mu=(\mu^g,\mu^h,\mu^G,\mu^H)$ satisfying \eqref{EqFOSCMS1}-\eqref{EqFOSCMS3},

    \begin{eqnarray}
    \label{EqTanLambda1}&&\mu^g_i\geq 0,\ i\in I_g(u):\lambda_i^g=0, \quad \mu^g_i= 0,\ i\not\in I_g(u),\\
    \label{EqTanLambda2}&&\mu^H_i=0,\ i\in \bar I^{0+}\cup I^{0+}(u),\quad\mu^G_i=0,\ i\in \bar I^{+0}\cup I^{+0}(u),\\
    \label{EqTanLambda3}&&\mbox{either $\mu^G_i>0,\mu^H_i>0$ or $\mu^G_i\mu^H_i=0$},\ i\in I^{00}(u):\lambda^G_i=\lambda^H_i=0\\
    \label{EqTanLambda4}&&\mu^G_i\geq0,\ i\in I^{00}(u):\lambda^G_i=0, \lambda^H_i>0,\quad \mu^G_i=0,\ i\in I^{00}(u):\lambda^G_i=0, \lambda^H_i<0,\\
    \label{EqTanLambda5}&&\mu^H_i\geq0,\ i\in I^{00}(u):\lambda^H_i=0, \lambda^G_i>0,\quad \mu^H_i=0,\ i\in I^{00}(u):\lambda^H_i=0, \lambda^G_i<0,\quad\\
    \label{EqGradLagrZero}&&\nabla_x{\cal L}^1(\xb,\lambda,\ob)=0,\\
    \label{EqSecOrderStab}&&\nabla_x^2{\cal L}^1(\xb,\lambda,\ob)u+\nabla_x{\cal L}^0(\xb,\mu,\ob)=0,
    \end{eqnarray}
    then there are neighborhoods $U$ of $\xb$, $W$ of $\ob$ and a constant $L>0$ such that for the stationary solution mappings $S(\omega)$ and $S_M(\omega)$ of $MPEC(\omega)$ one has
    \[(S(\omega)\cup S_M(\omega))\cap U\subset\xb +L\tau_1(\omega)\B_{\R^n}\ \forall\omega\in W.\]
\item If property $R_2$ holds and there does not exist $0\not=u\in T^{\rm lin}_\ob(\xb)$, $\lambda=(\lambda^g,\lambda^h,\lambda^G,\lambda^H)$, $\mu=(\mu^g,\mu^h,\mu^G,\mu^H)$ and $\nu\in\R^n$ satisfying \eqref{EqFOSCMS1}-\eqref{EqFOSCMS3}, \eqref{EqTanLambda1}-\eqref{EqSecOrderStab} and
    \begin{eqnarray}
    \label{EqTanNu1}\hspace{-3em}&&\nabla_x g_i(\xb,\ob)\nu+ u^T\nabla _x^2g_i(\xb,\ob)u\leq 0,\ i\in\bar I_g,\\
    \label{EqTanNu2}\hspace{-3em}&&\nabla_x h_i(\xb,\ob)\nu+ u^T\nabla _x^2h_i(\xb,\ob)u= 0,\ i=1,\ldots, m_E,\\
    \label{EqTanNu3}\hspace{-3em}&&\nabla_x G_i(\xb,\ob)\nu+ u^T\nabla _x^2G_i(\xb,\ob)u=0,\ i\in\bar I^{0+}\cup I^{0+}(u),\\
    \label{EqTanNu4}\hspace{-3em}&&\nabla_x H_i(\xb,\ob)\nu+ u^T\nabla _x^2H_i(\xb,\ob)u=0,\ i\in\bar I^{+0}\cup I^{+0}(u),\\
    \label{EqTanNu5}\hspace{-3em}&&-(\nabla_x G_i(\xb,\ob)\nu+ u^T\nabla _x^2G_i(\xb,\ob)u,\nabla_x H_i(\xb,\ob)\nu+ u^T\nabla _x^2H_i(\xb,\ob)u)\in T((0,0),Q_{EC}),\ i\in I^{00}(u),\qquad\\
    \label{EqSecOrderNu}\hspace{-3em}&&u^T\nabla_x^2{\cal L}^0(\xb,\lambda,\ob)u=\nabla_x f(\xb,\ob)\nu,
    \end{eqnarray}
    then there are neighborhoods $U$ of $\xb$, $W$ of $\ob$ and a constant $L>0$ such that for the stationary solution mappings $S(\omega)$ and $S_M(\omega)$ of $MPEC(\omega)$ one has
    \[(S(\omega)\cup S_M(\omega))\cap U\subset\xb +L\tau_2(\omega)\B_{\R^n}\ \forall\omega\in W.\]
\end{enumerate}
\end{corollary}
In the following table we specify for each condition of Theorem \ref{ThStab} its corresponding counterpart in Corollary  \ref{CorStabMPEC}:

\[\begin{tabular}{|c|c||c|c|}
\noalign{\hrule}
Thm.\ref{ThStab}(1.)& Cor.\ref{CorStabMPEC}(1.)&Thm.\ref{ThStab}(2.)& Cor.\ref{CorStabMPEC}(2.)\\
\noalign{\hrule}
\eqref{EqLipCond_lambda}&\eqref{EqFOSCMS1}-\eqref{EqFOSCMS3}&\eqref{EqHoeldCond1}&\eqref{EqTanNu1}-\eqref{EqTanNu5}\\
\eqref{EqLipCond_mu}&\eqref{EqTanLambda1}-\eqref{EqTanLambda5}&\eqref{EqHoeldCond2}&\eqref{EqSecOrderNu}\\
\eqref{EqLipCond_FOEq}&\eqref{EqGradLagrZero}&&\\
\eqref{EqLipCond_SOEq}&\eqref{EqSecOrderStab}&&\\
\noalign{\hrule}
\end{tabular}\]

For verifying RSSOSC we also need the sets $\Lambda^1_\ob(\xb;u)$. For given $u\in{\cal C}(\xb)$,  this set is the collection of all multipliers $\lambda=(\la^g,\la^h, \la^G,\la^H)\in \R^{m_I}_+ \times \R^{m_E} \times \R^{m_C}\times \R^{m_C}$ fulfilling \eqref{EqFOSCMS1}-\eqref{EqFOSCMS3} and $\nabla_x {\cal L}^{1}(x,\la,\ob) = 0$.

Finally let us mention that the definition of M-stationary solutions for $MPEC(\omega)$ is the same as coined by Scholtes \cite{Sch01},
i.e., the set of MPEC-multipliers associated with $x$,
\begin{eqnarray*}
\lefteqn{\Lambda_{MPEC}^M(x,{\omega})}\\& := &\left\{ \lambda=(\la^g,\la^h, \la^G,\la^H)\in \R^{m_I}_+ \times \R^{m_E} \times \R^{m_C}\times \R^{m_C}\mv
\begin{array}{l}\nabla_x {\cal L}^{1}(x,\la,\omega) = 0,\\
{\lambda^g}^Tg(\xb,\bar\omega) = 0,\\
\lambda^H_i=0,\ i\in \bar I^{0+},\quad\lambda^G_i=0,\ i\in \bar I^{+0},\\
\mbox{either $\lambda^G_i>0,\lambda^H_i>0$ or $\lambda^G_i\lambda^H_i=0$},\ i\in I^{00}
\end{array}\right\},,
\end{eqnarray*}
is not empty.

Now we demonstrate  our results in the following examples:
\begin{example}
\label{ExDemoMPEC}
Consider the following MPEC depending on the parameter $\omega\in\Omega:=\R_+$,
\begin{eqnarray*}
MPEC(\omega)\qquad \min_x -2x_1+x_2&&\\
\mbox{subject to}\qquad g_1(x,\omega):=x_1-x_2-\omega&\leq& 0,\\
g_2(x,\omega):=\omega \phi(x_1)-x_2&\leq&0,\\
-(G_1(x,\omega),H_1(x,\omega)):=-(x_1,x_2)&\in& Q_{EC},
\end{eqnarray*}
where $\phi(x):=x^6(1-\cos\frac 1{x})$ for $x\not=0$ and $\phi(0):=0$, with $\xb=(0,0)$ and $\ob=0$. Then we have $\bar I_g=\{1,2\}$, $\bar I^{0+}=\bar I^{+0}=\emptyset$, $\bar I^{00}=\{1\}$,
\[T_\ob^{\rm lin}(\xb)=\left\{(u_1,u_2)\mv u_1-u_2\leq 0,\ -u_2\leq 0,\ -(u_1,u_2)\in Q_{EC}
\right\}=\{(0,u_2)\mv u_2\geq 0\}\]
and we obtain, that $I_g(u)=I^{+0}(u)=I^{00}(u)=\emptyset$, $I^{0+}(u)=\{1\}$ for every $0\not= u=(0,u_2)\in T_\ob^{\rm lin}(\xb)$  and the set of multipliers $\lambda=(\lambda^g_1,\lambda^g_2,\lambda^G,\lambda^H)$ fulfilling \eqref{EqFOSCMS1}-\eqref{EqFOSCMS3}
is given by the relations
\begin{equation}\label{EqExampleCondLambda}\lambda^g_1=0,\ \lambda^g_2=0,\ \lambda^H=0.\end{equation}
Since \eqref{EqFOSCMS4} reads as
\[\lambda^g_1(1,-1)+\lambda^g_2(0,-1)-\lambda^G(1,0)-\lambda^H(0,1)=(0,0),\]
we conclude that FOSCMS is fulfilled and therefore property $R_1$ holds. But for $u=(0,0)$ we see that $\lambda=(0,1,0,-1)$ satisfies \eqref{EqFOSCMS1}-\eqref{EqFOSCMS4} and therefore metric regularity around $(\xb,0)$ does not hold.

It is easy to see that ${\cal C}_\ob(\xb)=\{0\}$. Hence RSSOSC is fulfilled, $\xb$ is an extended M-stationary point in the sense of \cite{Gfr14a} and therefore $\xb$ is an essential local minimizer of second order by \cite[Theorem 3.21]{Gfr14a}. Hence we can apply Theorem \ref{ThStabLocMin}(1.) to obtain
\[(S(\omega)\cup S_M(\omega))\cap U\subset L\omega \B_{\R^2}\ \forall\omega\in W\]
for some neighborhoods $U$ of $\xb$ and $W$ of $\ob$ and some constant $L>0$. Further we have $X(\omega)\cap U\not=\emptyset$ $\forall \omega\in W$, since $(0,1)\in T_\ob^{\rm lin}(\xb)$.
Further, the set of MPEC-multipliers is given by
\[\Lambda_{MPEC}^M(\xb,\ob)=\{(0,1,-2,0)\}\cup\{(2,\lambda_2^g,0,\lambda^H)\mv \lambda_2^g\geq 0, \lambda_2^g+\lambda^H=-1\}.\]

Now let us compute the sets $S(\omega)$, $S_M(\omega)$ and $X(\omega)$. For every $\omega> 0$ the feasible set ${\cal F}(\omega)$ consists of the union of the nonnegative $x_2$ axis, $\{(0,x_2)\mv x_2\geq 0\}$, and the set $\tilde X(\omega):=\{(x_1,0)\mv 0<x_1\leq\omega,  \phi(x_1)=0\}=\{(\frac 1{2k\pi},0)\mv k\in\N, 2k\pi\omega\geq 1\}$ consisting of a sequence of isolated points with limit $\xb$. It is easily checked that $(0,0)$ is M-stationary
with unique MPEC-multiplier $(0,1,-2,0)$
for every $\omega>0$, but not a local minimizer. Further, every element of $\tilde X(\omega)$ is a local minimizer and consequently B-stationary, but all these local minimizers are not M-stationary, because the constraints are degenerated; the only possible exception being the point $(\omega,0)$ if $\frac 1{2\pi\omega}\in \N$,
where the multipliers are
\[\Lambda_{MPEC}^M((\omega,0),\omega)=\{(2,\lambda_2^g,0,\lambda^H)\mv \lambda_2^g\geq 0, \lambda_2^g+\lambda^H=-1\}.\]
Summarizing all, we have
\[S(\omega)=X(\omega)=\tilde X(\omega), S_M(\omega)=\begin{cases}
\{(0,0)\}&\mbox{if $\frac 1{2\pi\omega}\not\in \N$}\\
\{(0,0), (\omega,0)\}&\mbox{else}
\end{cases}\]
for  $\omega>0$ and $S(0)=S_M(0)=X(0)=\{0,0\}$. It is quite surprising that we could prove existence and upper Lip\-schitz continuity of solutions in the absence of metric regularity of the constraints, a situation which is not possible in case of NLP (see,e.g. \cite[Lemma 8.31]{KlKum02}).

For the sake of completeness we also formulate explicitly  the sufficient conditions for upper Lip\-schitz continuity of Corollary \ref{CorStabMPEC}, although we know by the proof of Theorem \ref{ThStabLocMin} that they are implied by RSSOSC. The conditions of Corollary \ref{CorStabMPEC} are, that there are no elements $u=(0,u_2)$ with $u_2>0$ and $\lambda=(\lambda^g_1,\lambda^g_2,\lambda^G,\lambda^H)$, $\mu=(\mu^g_1,\mu^g_2,\mu^G,\mu^H)$ fulfilling \eqref{EqExampleCondLambda},
$\mu^g_1=\mu^g_2=\mu^H=0$,
\begin{equation}\label{EqExampleGradLagr}(-2,1)+\lambda^g_1(1,-1)+\lambda^g_2(0,-1)-\lambda^G(1,0)-\lambda^H(0,1)=(0,0)
\end{equation}
and
\[(0,0)+\mu^g_1(1,-1)+\mu^g_2(0,-1)-\mu^G(1,0)-\mu^H(0,1)=(0,0).\]
It is easy to see that this holds true, because \eqref{EqExampleCondLambda} and \eqref{EqExampleGradLagr} are inconsistent.

For this example let us now compare  our results with those of Guo,Lin and Ye \cite[Theorem 5.5]{GuLiYe12}, who considered a somewhat different setting, namely the calmness of pairs of M-stationary solutions together with their associated multipliers. Application of \cite[Theorem 5.5]{GuLiYe12} yields that, for every $\lab\in\Lambda_{MPEC}(\xb,\ob)$, there are neighborhoods $U$ of $(\xb,\lab)$ and $W$ of $\ob$ and a constant $\kappa>0$ such that for every $(x,\lambda,\omega)\in U\times W$ with $x\in S_M(\omega)$, $\lambda\in \Lambda_{MPEC}^M(x,\omega)$ one has
\[\norm{x-\xb}+\dist{\lambda,\Lambda_{MPEC}(\xb,\ob)}\leq \kappa\norm{\omega -\ob}.\]
In particular, by applying this result with $\lab=(0,1,-2,0)$ and $\lab=(2,0,0,-1)$, respectively, we obtain that
$S_M(\omega)$ behaves upper Lipschitz stable. But in our example we also have local minimizers in $X(\omega)$, which are not M-stationary, and \cite[Theorem 5.5]{GuLiYe12} does not give any information for these local minimizers, whereas our theory establishes upper Lipschitz continuity.
Moreover, if we slightly change the data of this example by setting $\phi(x):=x^6(0.9-\cos\frac 1{x})$ for $x\not=0$ and $\phi(0):=0$, then one can show that for $\omega>0$ all local minimizers in $X(\omega)$ are also M-stationary, but for all local minimizers different from $(\omega,0)$ one has that the corresponding MPEC multiplier tends to infinity when $\omega$ tends to $\ob$. Also in this case \cite[Theorem 5.5]{GuLiYe12} does not give any information due to the discontinuity of the multipliers in contrast to our results.
\end{example}
\begin{example}\label{JoShikStef} [see \cite{JoShikStef12}]
Consider the problem
\begin{eqnarray*}
MPEC(\omega)\qquad &&\min_x-x_1-(x_2+\omega)^2\\
\mbox{subject to}\qquad&&
-(G_1(x,\omega),H_1(x,\omega)):=-(x_1,x_2)\in Q_{EC},
\end{eqnarray*}
depending on the parameter $\omega\in\Omega:=\R$. This example was used in \cite{JoShikStef12} to demonstrate that M-stationary points are not strongly stable in the sense of Kojima \cite{Koj80}, but the C-stationary points are strongly stable.
Straightforward analysis yields that
\[X(\omega)=S(\omega)=S_M(\omega)=\begin{cases}\emptyset&\mbox{if $\omega>0$,}\\
(0,-\omega)&\mbox{if $\omega<0$}
\end{cases}\]
and $X(0)=S(0)=\emptyset$, $S_M(0)=(0,0)$.
Let us now choose $\xb=(0,0)$ and $\ob=0$ as reference point and apply Corollary \ref{CorStabMPEC}(1.). Obviously, the constraint mapping $M_\ob$ is metrically regular around $(\xb,0)$ and thus property $R_1$ is fulfilled. Since the cone $ T^{\rm lin}_\ob(\xb)$ is generated by the 2 directions $u^1=(1,0)$ and $u^2=(0,1)$, in order to establish upper Lipschitz stability of $S(\omega)\cup S_M(\omega)$ it suffices  to show that there are no multipliers $\lambda=(\lambda^G,\lambda^H)$ and $\mu=(\mu^G,\mu^H)$ fulfilling either
\begin{eqnarray*}
&&\lambda^G=\mu^G=0,\\
&&\nabla_x{\cal L}^1(\xb,\lambda,\ob)=(-1,0)-\lambda^G(1,0)-\lambda^H(0,1)=0,\\
&&\nabla_x^2{\cal L}^1(\xb,\lambda,\ob)u^1+\nabla_x{\cal L}^0(\xb,\mu,\ob)=(0,0)-\mu^G(1,0)-\mu^H(0,1)=0
\end{eqnarray*}
or
\begin{eqnarray*}
&&\lambda^H=\mu^H=0,\\
&&\nabla_x{\cal L}^1(\xb,\lambda,\ob)=(-1,0)-\lambda^G(1,0)-\lambda^H(0,1)=0,\\
&&\nabla_x^2{\cal L}^1(\xb,\lambda,\ob)u^2+\nabla_x{\cal L}^0(\xb,\mu,\ob)=(0,-2)-\mu^G(1,0)-\mu^H(0,1)=0,
\end{eqnarray*}
which is obviously the case. Hence we obtain
\[(S(\omega)\cup S_M(\omega))\cap U\subset L\omega \B_{\R^2}\ \forall\omega\in W\]
for some neighborhoods $U$ of $\xb$ and $W$ of $\ob$ and some constant $L>0$.\\
Note that the assumptions of \cite[Theorem 5.5]{GuLiYe12} are not fulfilled, and therefore it cannot be applied.
\end{example}

\noindent Corollary \ref{CorStabMPEC} recovers well-known results for standard nonlinear programs in the case of $C^2$ data.
Consider the parametric nonlinear optimization problem
\begin{equation}\label{nlp}
NLP(\omega)\qquad \min f(x,\omega) \quad \mbox{subject to} \quad
g(x,\omega)\leq 0,\ h(x,\omega)=0,
\end{equation}
where the function $f:\R^n\times\R^s\to\R$, and $q=(g,h): \R^n \times\R^s\to \R^{m_I}\times \R^{m_E}$
satisfy the Basic Assumption \ref{AssBasic} and, in addition,  that $f(\xb,\cdot)$, $\nabla_x f(\xb,\cdot)$, $q(\xb,\cdot)$ and $\nabla_x q(\xb,\cdot)$ are Lipschitzian near $\ob$.

We use the same notation as above in this section,
the results and representations reduce to forms which omit the complementarity data $G_i$ and $H_i$.
In particular, the cone $T_\ob^{\rm lin}(\xb)$ is given by
\[T_\ob^{\rm lin}(\xb)=\left\{u\in\R^n\mv \begin{array}{l}\nabla_x g_i(\xb,\ob)u\leq 0,\ i\in\bar I_g,\\
\nabla_x h_i(\xb,\ob)u=0,\ i=1,\ldots,m_E
\end{array}\right\},
\]
with $\bar{I}_g = \{ i \in \{1,\ldots,m_I\} \mv g_i(\xb,\ob) = 0\}$, and the extended Lagrangian defined above reduces to
\[
{\cal L}^{\lambda_0}(x,\lambda^g,\la^h,\omega):=\lambda_0 f(x,\omega)+\sum_{i=1}^{m_I}\lambda_i^g g_i(x,\omega)+\sum_{i=1}^{m_E}\lambda_i^h h_i(x,\omega).
\]
Then a point $x\in {\cal F}(\omega)$ feasible for $NLP(\omega)$ is called \textit{stationary in the
Karush-Kuhn-Tucker (KKT) sense}, if the set of multipliers associated with $x$,
\[
\Lambda(x,{\omega}) := \{ (\la^g,\la^h)\in \R^{m_I}_+ \times \R^{m_E} \, | \,
\nabla_x {\cal L}^{1}(x,\la^g,\la^h,\omega) = 0, {\lambda^g}^Tg(\xb,\bar\omega) = 0 \}
\]
is not empty. Further recall that under some constraint qualification (CQ), e.g. the Mangasarian-Fromovitz CQ (MFCQ),
the three
concepts of M-stationarity, B-stationarity and KKT-stationarity coincide. Since MFCQ persists under small perturbations
in our setting, hence $S(\omega) \cap U$ and $S_M(\omega) \cap U$ coincide for some neighborhood $U$ of $\xb \in S(\ob)$ and all $\omega$ close to $\ob$, provided that $\xb$ satisfies MFCQ.

We show how to prove two classical results by means of Corollary \ref{CorStabMPEC} or Theorem \ref{ThStabLocMin}, respectively.
\begin{corollary} \label{nlpul} (Klatte, Kummer \cite[Theorem 8.24]{KlKum02}) \ Given a stationary solution
$\xb$ of $NLP(\bar{\omega})$ in the KKT sense,
we suppose that MFCQ is satisfied at $(\xb,\bar\omega)$. Then $S$ is locally upper Lipschitz at   $(\bar\omega,\xb)$ if
for every $\la = (\la^g, \la^h) \in \Lambda(\xb,\bar{\omega})$, the system
\begin{equation}\label{cstab}
\begin{array}{l}
\, \  u \in T_\ob^{\rm lin}(\xb), \quad  (\alpha,\beta) \in \R^{m_I}\times \R^{m_E},  \vspace{1ex} \\
\, \   \nabla_x^2{\cal L}^1(\xb,\la^g,\la^h,\ob)u+\nabla_x{\cal L}^0(\xb,\alpha,\beta,\ob)=0 ,  \vspace{1ex}\\
\begin{array}{ll}
\nabla_x g_i(\xb,\ob) \,u = 0, & \mbox{if } \ i \in \bar{I}_g: \ \la^g_i > 0, \vspace{1ex} \\
\alpha_i \ge 0, \ \alpha_i \nabla_x g_i(\xb,\ob)\, u = 0 , & \mbox{if }  \ i \in \bar{I}_g:  \ \la^g_i = 0,\vspace{1ex}\\
\alpha_i = 0 , & \mbox{if } \ i \not \in  \bar{I}_g.
\end{array}
\end{array}
\end{equation}
has no solution $(u,\alpha,\beta)$ with $u \not= 0$.
\end{corollary}
\begin{proof}
Take any $(u,\la,\mu)$ which satisfies  $u \in  T_\ob^{\rm lin}(\xb)$, (\ref{EqFOSCMS1}),   (\ref{EqTanLambda1}),   (\ref{EqGradLagrZero}) and  (\ref{EqSecOrderStab}). It is sufficient to show that  $(u,\la,\mu)$ also satisfies both
$\la \in  \Lambda(\xb,\bar{\omega})$ and (\ref{cstab})  when putting $\alpha=\mu^g$ and $\beta =\mu^h$. Indeed,
then the assumption of the corollary says that $u = 0$. Since MFCQ at $(\xb,\bar\omega)$ implies property $R_1$,
Corollary \ref{CorStabMPEC} thus immediately gives in the special case $NLP(\omega)$
\[
S(\omega) \cap U \subset\xb +L \tau_1(\omega)\ \forall \omega\in W.
\]
By our assumptions for $f$ and  $q=(g,h)$, we have for all  $\omega$ sufficiently close to $\bar\omega$,
\begin{eqnarray*}
\tau_1(\omega) & = &\norm{\nabla_x f(\xb,\omega)-\nabla_x f(\xb,\ob)}
+ \norm{\nabla_x q(\xb,\omega)-\nabla_x q(\xb,\bar\omega)}+{\norm{q(\xb,\omega)-q(\xb,\bar\omega)}}\\
& \le & const. \, \| \omega - \bar\omega \|
\end{eqnarray*}
and so we are done.

It remains to show that $(u,\la,\mu)$ satisfies both $\la \in  \Lambda(\xb,\bar{\omega})$ and (\ref{cstab})  when putting $\alpha=\mu^g$ and $\beta =\mu^h$. First, we observe that  (\ref{EqFOSCMS1}) says that $\la_i^g = 0$ if $i \not\in \bar{I}_g$ or if $i \in \bar{I}_g \setminus I_g(u)$, and $\la_i^g \ge 0$ if $i \in I_g(u)$. Together with (\ref{EqGradLagrZero}), this gives $\la \in \Lambda(\xb,\bar{\omega})$.
Second, $u \in T_\ob^{\rm lin}(\xb)$ and (\ref{EqSecOrderStab}) directly appear in the first two lines of (\ref{cstab}).
Moreover,  (\ref{EqTanLambda1}) says (i) $\alpha_i= \mu_i^g = 0$ if $i \not \in \bar{I}_g$ or if $i \in \bar{I}_g: \ \nabla_x g_i(\xb,\ob) \,u < 0$, and (ii) $\alpha_i = \mu_i^g \ge 0$ if $ i \in \bar{I}_g:  \ \la^g_i = 0$, which implies that $(u,\alpha)$  fulfills the last two lines of  (\ref{cstab}).
Finally, assuming that for some  $i \in \bar{I}_g: \ \la^g_i > 0$ one has $\nabla_x g_i(\xb,\ob) \,u \not= 0$,
i.e., $i \not \in I_g(u)$, this contradicts  (\ref{EqFOSCMS1}).  Therefore, also the third line of  (\ref{cstab})  is satisfied.
\qed\end{proof}
\begin{remark}\label{remnlpul} The opposite direction of Corollary \ref{nlpul} is true if the parametric problem (\ref{nlp}) includes canonical perturbations,
i.e. if $f$ and $q=(g,h)$ are defined by $f(x,\omega) =  \tilde{f}(t,x)  - \langle a,x \rangle$ and
$(g,h)(x,\omega) =   (\tilde{g},\tilde{h})(t,x)  - b$  for varying $\omega = (t,a,b) \in \Omega \times \R^n \times \R^{m_I+m_L}$,
see \cite[Thm. 8.24]{KlKum02}.
\end{remark}
The following H\"older stability result is a generalization of Proposition 4.41 in Bonnans and Shapiro \cite{BonSh00},
where global minimizing sets of $NLP(\omega)$ are considered instead of the sets of stationary solutions or local minimizers as
here. Recall that  $\xb \in \cal F (\bar\omega)$ satisfies the quadratic growth condition for
$NLP(\bar\omega)$ if $f(x,\bar\omega)-f(\xb,\bar\omega) \geq \eta\norm{x-\xb}^2\ \forall x\in V \cap {\cal F}(\bar\omega)$ holds for some $\eta > 0$ and some neighborhood $V$ of $\xb$.
\begin{corollary}\label{hoeldmin} (upper H\"older continuity of local minimizers)\\
Suppose that $\xb \in \cal F (\bar\omega)$ satisfies the quadratic growth condition for
$NLP(\bar\omega)$, MFCQ is satisfied at $\xb$, and the functions $f$, $g$, and $h$  are twice continuously differentiable near $(\xb,\bar \omega)$.
Then $X$ and $S$ are locally nonempty-valued and upper H\"older of order $\frac 12$  at   $(\bar\omega,\xb)$,
i.e.,  there are neighborhoods $U$ of $\xb$, $W$ of $\ob$ and a constant $L>0$ such that for the solution mapping $X(\omega)$ of $P(\omega)$ one has
\[
   \emptyset \not=  X(\omega)\cap U
    \subset S(\omega)  \cap U
  \subset  \{\xb\}  +L \| \omega -\bar\omega \|^{\frac 12} \quad \forall\omega\in W.
\]
\end{corollary}
\begin{proof} The corollary is an immediate consequence of the second part of Theorem \ref{ThStabLocMin} by taking into account that MFCQ is equivalent to metric regularity of the multifunction $M_\ob$ around $(\xb,0)$ implying SOSCMS by virtue of Theorem \ref{ThCharMetrSubReg},
and that under MFCQ the quadratic growth condition is equivalent to the property that $\xb$ is an essential local minimizer of second-order.
\qed\end{proof}
Results of this type are classical, but concern global or so-called complete local minimizing (CLM) sets, see e.g.
\cite{Alt83,AruIz05,BonSh00,Kla94}. Corollary \ref{hoeldmin} extends this by providing even existence and upper H\"older continuity for stationary solution sets and the sets of \textit{all} local minimizers, while part 2 of Theorem \ref{ThStabLocMin} allows us in addition to weaken the assumption MFCQ by assuming SOSCMS (which implies $R_2$) instead. Note that
upper H\"older stability of the global minimizing or CLM set mapping even holds if $\xb$ is not locally isolated (see e.g. \cite{BonSh00,Kla94}). Recently, Kummer \cite{Kum09} presented an alternative characterization of upper H\"older stability of KKT- stationary solutions via convergence properties of suitable iteration procedures.

\begin{remark}\label{remhoeldmin}
For nonlinear programs, a standard second order sufficient optimality condition at
$\xb \in {\cal F}(\ob)$ (see for example \cite[Prop.5.48]{BonSh00}) is equivalent to the property \eqref{EqEssLocMin} (i.e., $\xb$ is an essential local minimizer of second order for the problem $NLP(\bar\omega)$). This is true without any constraint qualification, see the remarks following Thm. 5.11 in \cite{Gfr06}.

Further note that in the case (\ref{nlp}), SOSCMS  is fulfilled for $M_\ob$ at $\xb$ if and only if for every direction $0\not=u\in T_\ob^{\rm lin}(\xb)$,
the only multiplier $\lambda=(\lambda^g,\lambda^h)\in\R^{m_I}\times \R^{m_E}$ satisfying
\begin{eqnarray*}
&&\sum_{i=1}^{m_I}\lambda_i^g\nabla g_i(\xb,\ob)+\sum_{i=1}^{m_E}\lambda_i^h\nabla h_i(\xb,\ob)=0,\ \lambda^g_i\geq 0,\ i\in I_g(u),\ \lambda^g_i= 0,\ i\not\in I_g(u),\\
&&\\
&&u^T\left(\sum_{i=1}^{m_I}\lambda_i^g\nabla_x^2 g_i(\xb,\ob)+\sum_{i=1}^{m_E}\lambda_i^h\nabla_x^2 h_i(\xb,\ob)\right)u\geq 0
\end{eqnarray*}
is $\lambda=0$, where as above
$I_g(u)=\{i\in \bar{I}_g \mv  \nabla_x g_i(\xb,\ob)u=0\}$.
\end{remark}
\begin{example} \label{ExNLP} This example shows, as mentioned above, that part 2 of Theorem \ref{ThStabLocMin} gives a stronger result than the (more classical) Corollary  \ref{hoeldmin}.  Consider the parameter dependent program
\[
P(\omega) \quad \min\, x_1^2 - x_2^2 \quad \mbox{s.t.} \ -x_2 + \omega \le 0, \ x_2 - x_1^2 \le 0,
\]
with $\Omega = \R$, $\ob=0$ and $\xb=(0,0)$. $M_\ob$ is not metrically regular near $(\xb,0)$. One easily sees that SOSCMS holds for $M_\ob$ at $\xb$: one has to check only the directions $u = (\pm 1,0) \in T_\ob^{\rm lin}(\xb)$. Obviously, $\xb$ is an
an essential local minimizer of second order for the problem $P(\bar\omega)$, and one has
\[
S(\omega) = S_M(\omega) = \left\{ \begin{array}{ll}
\{(0,0), (0,\omega)\} & \mbox{if } \omega < 0,\\
\{(\pm \sqrt{\omega},\omega)\} & \mbox{if } 0 \le \omega \le \frac12 ,
\end{array} \right.
\]
\[
X(\omega) = \left\{ \begin{array}{ll}
\{(0,\omega)\} & \mbox{if } \omega < 0,\\
\{(\pm \sqrt{\omega},\omega)\} & \mbox{if } 0 \le \omega \le \frac12 .
\end{array} \right.
\]
Indeed, the assumptions and the statement of part 2 of Theorem \ref{ThStabLocMin}  are fulfilled.
\end{example}

\if{At the end of this paper we want to present an example which demonstrates, that our results incorporate the combinatorial structure of the problem and are more far reaching than those results which one could obtain by decomposing into different branches  and applying direct arguments to the subproblems.}\fi
At the end of this paper let us mention, that our results incorporate the combinatorial structure of the problem and are more far reaching than those results which one could obtain by decomposing into different branches  and applying direct arguments to the subproblems. One simple reason is that (globally optimal) solutions of the subproblems may not be M-stationary solutions of the overall problem, as the following example demonstrates.
\begin{example}
Consider the problem
\begin{eqnarray*}
MPEC(\omega)\qquad \min_x&& x_2\\
\mbox{subject to}\qquad g_1(x,\omega)&:=& x_1^2-x_2\leq0\\
-(G_1(x,\omega),H_1(x,\omega))&:=&-(x_2+x_3+\omega,x_3)\in Q_{EC},
\end{eqnarray*}
with $\xb=(0,0,0)$, $\ob=0$. Then we have
\[X(\omega)=S(\omega)=S_M(\omega)=\begin{cases}(0,0,0)&\mbox{if $\omega\geq 0$,}\\
(0,0,-\omega)&\mbox{if $\omega<0$.}\end{cases}\]
Since the only multiplier $\lambda=(\lambda^g,\la^G,\la^H)$ fulfilling
\[\la^g\geq 0, \mbox{ either $\la^G>0,\la^H>0$ or $\la^G\la^H=0$},\ \la^g(0,-1,0)-\la^G(0,1,1)-\la^H(0,0,1)=0\]
is $\la=(0,0,0)$, the constraint mapping $M_\ob$ is metrically regular around $(\xb,0)$ by Theorem \ref{ThCharMetrSubReg}. Consequently FOSCMS and property $R_1$ are fulfilled.
Now let us verify that RSSOSC also holds. The condition $u=(u_1,u_2,u_3)\in{\cal C}_\ob(\xb)$ amounts to
\[u_2\leq 0,\ -u_2\leq 0, -(u_2+u_3)\leq 0,\ -u_3\leq 0, (u_2+u_3)u_3=0\]
or equivalently $u_2=u_3=0$. Hence, for every $u\in {\cal C}_\ob(\xb)$, the set $\Lambda^1_\ob(\xb;u)$ is the collection of all $\lambda=(\lambda^g,\Lambda^G,\lambda^H)$ fulfilling
\begin{eqnarray*}&&(0,1,0)+\lambda^g(0,-1,0)-\lambda^G(0,1,1)-\lambda^H(0,0,1)=0\\
&&\lambda^g\geq 0,\ \mbox{either $\lambda^G,\lambda^H>0$ or $\lambda^G\lambda^H=0$,}
\end{eqnarray*}
resulting in $\Lambda^1_\ob(\xb;u)=\{(1,0,0)\}$. Hence, $u^T\nabla^2_x{\cal L}^1(\xb,\la,\ob)u=2u_1^2>0$ $\forall 0\not=u\in {\cal C}_\ob(\xb), \la\in\Lambda^1_\ob(\xb;u)$, showing the validity of RSSOSC and thus, by virtue of Theorem \ref{ThStabLocMin},
\[(S(\omega)\cup S_M(\omega))\cap U\subset L\omega \B_{\R^3}\ \forall\omega\in W\]
for some neighborhoods $U$ of $\xb$ and $W$ of $\ob$ and some constant $L>0$. We see that by our results we can establish upper Lipschitz stability of  B-stationary and M-stationary solutions and hence, in particular, also of local minimizers.\\
Now we want to compare this result with the one, which one would obtain by decomposing into different branches.
Since $Q_{EC}=(\R_-\times \{0\})\cup (\{0\}\times \R_-)$, the MPEC decomposes into the two nonlinear convex programs
\begin{eqnarray*}
NLP^1(\omega)\qquad \min_x&& x_2\\
\mbox{subject to}\qquad g_1(x,\omega)&:=& x_1^2-x_2\leq0,\\
g_2(x,\omega)&:=&-(x_2+x_3+\omega)\leq 0,\\
h_1(x,\omega)&:=& -x_3=0.
\end{eqnarray*}
and
\begin{eqnarray*}
NLP^2(\omega)\qquad \min_x&& x_2\\
\mbox{subject to}\qquad g_1(x,\omega)&:=& x_1^2-x_2\leq0,\\
g_2(x,\omega)&:=& -x_3\leq 0,\\
h_1(x,\omega)&:=&-(x_2+x_3+\omega)= 0.
\end{eqnarray*}
For the first problem the set of  minimizers is given by
\[X^1(\omega)=\begin{cases}\{(0,0,0)\}&\mbox{if $\omega\geq0$,}\\
\{(x_1,-\omega,0)\mv \vert x_1\vert\leq \sqrt{-\omega}\}&\mbox{if $\omega<0$}
\end{cases}\]
and for the second problem we obtain
\[X^2(\omega)=\begin{cases}\emptyset&\mbox{if $\omega>0$,}\\
\{(0,0,-\omega)\}&\mbox{if $\omega\leq0$}
\end{cases}\]
Hence, for the union $X^1(\omega)\cup X^2(\omega)$ we only have upper H\"older continuity and this demonstrates that  in a straightforward way we cannot derive  the stability property of solutions of the MPEC from the continuity properties of solutions of the different branches.
\end{example}
\section*{Acknowledgements}
The authors are very grateful to
the referees for constructive comments that significantly improved the presentation. In particular, one referee suggested to use the  assumption \eqref{EqSuffCondR2} in Proposition \ref{PropPropertyR}(2.) instead of the stronger assumption of SOSCMS.

The research of the first author was supported by the Austrian Science Fund (FWF) under grant P26132-N25.

\end{document}